\documentclass[11pt,letterpaper]{amsart}
\usepackage{amsfonts,amsmath,amsthm,amssymb}

\usepackage{graphicx}
\usepackage{setspace}
\usepackage{thmtools}

\usepackage{dsfont}
\usepackage{grffile}

\usepackage [english]{babel}
\usepackage [autostyle, english = american]{csquotes}
\MakeOuterQuote{"}

\numberwithin{equation}{section}

\usepackage{xcolor}
\usepackage{hyperref}
\hypersetup{
    colorlinks = true,
    linkcolor = {blue}, urlcolor={blue}
}

\definecolor{c1}{rgb}{0,0,1} 
\hypersetup{
    linkcolor= {c1}, 
    citecolor={c1}, 
    urlcolor={c1} 
}

\newtheorem{thm}{Theorem}

\newtheorem{lemma}{Lemma}

\numberwithin{equation}{section}
\numberwithin{thm}{section}
\numberwithin{lemma}{section}
\numberwithin{cor}{section}
\numberwithin{prop}{section}

\begin{document}

{
  \title{\bf Nonlocal correlation in L-functions}
  \author{Gordon Chavez}

\date{}
  \maketitle
}

\begin{abstract}
We identify a nonlocal correlation structure in L-functions. This structure involves very long and infinite-range correlations between values of logarithmic L-functions, where the correlation strongly depends upon the presence of a multiplicative relationship between the two points in question on the complex plane. This correlation sharply jumps if the two points share a multiplicative relationship and takes much lower values otherwise, demonstrating a previously undescribed type of long-range order in L-functions. We leverage this correlation structure to provide a novel set of identities that relate the distribution of M{\"o}bius function solutions, sums over the non-divisors of the integers, and the polylogarithms. 
\end{abstract}


\tableofcontents

\section{Introduction}
The famous Riemann zeta function $\zeta(s)$ may be represented by the following sum over integers $n$ or product over prime numbers $p$ for $\textnormal{Re}(s)>1$:
\begin{equation}
\zeta(s)=\sum_{n=1}^{\infty}\frac{1}{n^{s}}=\prod_{p}\left(1-\frac{1}{p^{s}}\right)^{-1} \label{zeta}
\end{equation}
The former representation is called the Dirichlet series, while the latter representation is called the Euler product. $\zeta(s)$ may be extended by analytic continuation to all $s$ in $\mathbb{C}$ except $s=1$, as $\zeta(1)$ is a simple pole. L-functions $L\left(s,\chi\right)$ are an important generalization of $\zeta(s)$. Like $\zeta(s)$, these functions have a Dirichlet series and Euler product representation
\begin{equation}
L\left(s,\chi\right)=\sum_{n=1}^{\infty}\frac{\chi(n)}{n^{s}}=\prod_{p}\left(1-\frac{\chi(p)}{p^{s}}\right)^{-1}, \label{l function}
\end{equation}
where $\chi$ is a given Dirichlet character modulo $M$, which has the definition
\begin{align}
\chi(nm)=\chi(n)\chi(m), \nonumber \\
\chi(n)\begin{cases} =0; & \textnormal{gcd}(n,M)>1 \\ \neq 0; & \textnormal{gcd}(n,M)=1 \end{cases} \nonumber \\
\chi(n+M)=\chi(n) \label{dirichlet character}
\end{align}
If $\chi$ is a principal character, $\chi_{0}$, meaning it has the simple definition  
\begin{equation}
\chi_{0}(n)=\begin{cases} 0; & \textnormal{gcd}(n,M)>1 \\ 1; & \textnormal{gcd}(n,M)=1 \end{cases} \label{principal ch}
\end{equation}
then, similarly to $\zeta(s)$, $L(s,\chi_{0})$ may be extended by analytic continuation to all $s$ in $\mathbb{C}$ except $s=1$ where there is a simple pole. If $\chi$ is not principal, then $L(s,\chi)$ is an entire function.

The statistical properties of zeta and L-function values have been a subject of great interest since Bohr and Courant \cite{bohr} \cite{bohr courant} showed that, for any $1/2 <\sigma\leq 1$, the set of values taken by $\zeta(\sigma+it)$ as well as $\log\zeta(\sigma+it)$ under $t\in \mathbb{R}$ are dense in $\mathbb{C}$. This showed that the values of $\zeta(s)$ and $\log\zeta(s)$ are, in a precise sense, ergodic in the complex numbers. Decades later, Voronin \cite{voronin71} substantially extended Bohr's denseness result by showing that, for any $1/2<\sigma\leq 1$, the set of values taken by $\zeta(\sigma+it), \zeta'(\sigma+it),..., \zeta^{(m-1)}(\sigma+it)$, are dense in $\mathbb{C}^{m}$. He then proved a related and even more striking result: the universality theorem for the zeta function, which showed that $\zeta(s)$ and $\log\zeta(s)$ will approximate, to arbitrarily high accuracy, any non-vanishing holomorphic function in the critical strip \cite{voronin72}. Voronin \cite{voronin73} along with Gonek \cite{gonek} and Bagchi \cite{bagchi1} \cite{bagchi2} additionally proved that L-functions have the same and even more powerful universality properties, showing that L-functions are ergodic in the space of analytic functions. These denseness and universality results demonstrate that L-functions fluctuate very chaotically in the critical strip and exhibit an enormous variety of functional behavior. In this paper we show that, along with their chaotic fluctuations, there is also significant long-range statistical order in many of these functions. 

Below we describe very long and infinite-range correlations between values of logarithmic principal L-functions, where the correlation strongly depends upon the presence of a multiplicative relationship between the two points in question on the complex plane. In particular, we show that if $\chi_{0}$ is a principal Dirichlet character modulo $M$ and $p$ is the minimum prime $p$ such that $p \nmid M$, then, for two rational numbers $\alpha$ and $\beta$, and under uniformly distributed $t \in [a,b]$ with $b-a\rightarrow \infty$, the correlations 
\begin{equation}
\rho_{\chi_{0}}^{\textnormal{diag}}(\alpha,\beta,\sigma)=\textnormal{Corr}\left\{\log\left|L\left(\alpha\left(\sigma+it\right),\chi_{0}\right)\right|,\log\left|L\left(\beta\left(\sigma+it\right),\chi_{0}\right)\right|\right\} \label{gen diag corr intro}
\end{equation}
and 
\begin{equation}
\rho_{\chi_{0}}^{\textnormal{vert}}(\alpha,\beta,\sigma)=\textnormal{Corr}\left\{\log\left|L\left(\sigma+i\alpha t,\chi_{0}\right)\right|,\log\left|L\left(\sigma+i\beta t,\chi_{0}\right)\right|\right\}, \label{gen vert corr intro}
\end{equation}
which we refer to as \textit{diagonal} and \textit{vertical} correlations respectively, have the properties
\begin{equation}
\rho_{\chi_{0}}^{\textnormal{diag}}(\alpha,\beta,\sigma)=\begin{cases} 
       O\left(p^{-\alpha \sigma}\right) \hspace{.1cm}; & \beta|\alpha \\
      O\left(p^{-\left(2\beta-1\right)\alpha \sigma}\right) \hspace{.1cm}; & \beta \nmid \alpha
   \end{cases}
\label{nonloc diag corr 1}
\end{equation}
and
\begin{equation}
\rho_{\chi_{0}}^{\textnormal{vert}}(\alpha,\beta,\sigma)=\begin{cases} 
      O\left(p^{-\frac{\alpha}{\beta}\sigma}\right) \hspace{.1cm}; & \beta|\alpha \\
      O\left(p^{-\alpha\sigma}\right) \hspace{.1cm}; & \beta \nmid \alpha
   \end{cases} 
\label{nonloc vert corr 1}
\end{equation}
as $\alpha\rightarrow \infty$. This shows that, given two points on the complex plane, there can be high correlation between the values of logarithmic principal L-functions at these points, even if the points are separated by a very large or infinite distance. This correlation will make a sharp jump to much higher values if the two points have a multiplicative relationship and will otherwise take on much smaller values. This behavior is depicted for several vertical correlation functions in Figure \ref{big cor plot intro}. These results thus demonstrate a notable type of long-range order in L-functions that has not previously been described.

\begin{figure}
\centering
\includegraphics[width=\linewidth]{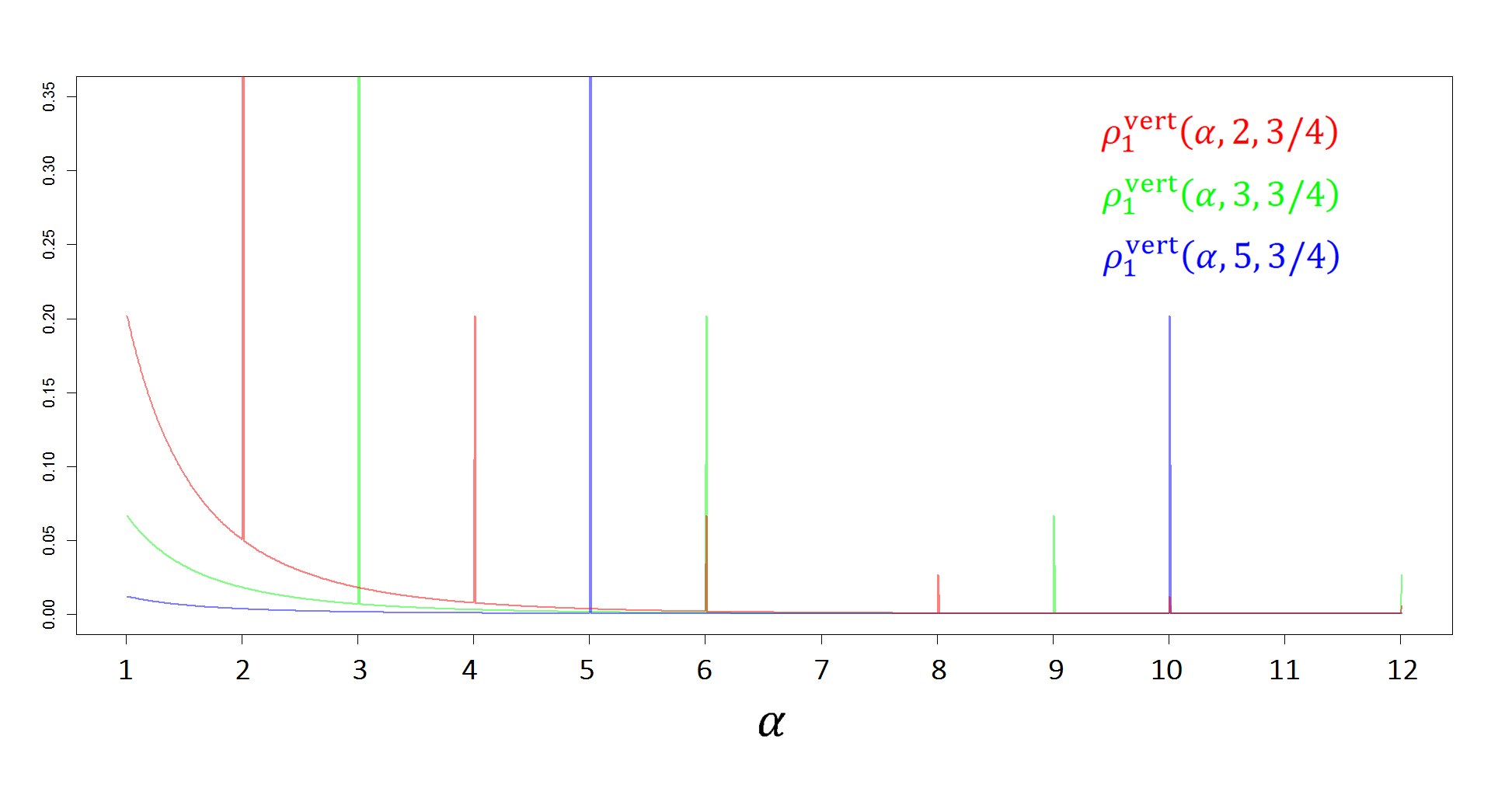}
\caption{Approximations of the correlations defined by (\ref{gen vert corr intro}) for the zeta function $\zeta(s)$ with $\textnormal{Re}(s)=3/4$, $\rho_{1}^{\textnormal{vert}}(\alpha,\beta,3/4)$, where $\beta=2$, $3$, $5$, and $1\leq \alpha \leq 12 \in \mathbb{Q}$.}
\label{big cor plot intro}
\end{figure}

We leverage this nonlocal correlation structure to prove the following set of identities:
\begin{thm}For all $v \in \mathbb{R}$ and  $x>2^{1/6}$,
\begin{equation}
\sum_{\substack{n>1 \\ \mu(n)=1}}\frac{1}{n^{v}}\textnormal{Li}_{v}\left(\frac{1}{x^{n}}\right)=\sum_{n}\frac{\mu(n)}{n^{v}}\sum_{\substack{m \nmid n \\ m<n}}\frac{\mu(m)}{m^{v}}\textnormal{Li}_{v}\left(\frac{1}{x^{nm}}\right) \label{dabigone}
\end{equation}
and 
\begin{equation}
\sum_{\substack{n >1 \\ \mu(n)=-1}}\frac{1}{n^{v}}\textnormal{Li}_{v}\left(\frac{1}{x^{n}}\right)=\textnormal{Li}_{v}\left(\frac{1}{x}\right)-\frac{1}{x}+\sum_{n}\frac{\mu(n)}{n^{v}}\sum_{\substack{m \nmid n \\ m<n}}\frac{\mu(m)}{m^{v}}\textnormal{Li}_{v}\left(\frac{1}{x^{nm}}\right). \label{dabigone2}
\end{equation}
\label{bigthm1}
\end{thm}
\hspace{-.45cm}The function $\mu(n)$ is the well-known M{\"o}bius function, defined as
\begin{equation}
\mu(n)=\begin{cases} 0; & \textnormal{$n$ has a repeated prime factor (not square-free)} \\ (-1)^{\omega(n)}; & \textnormal{$n$ has $\omega(n)$ distinct prime factors (square-free)} \end{cases} \label{mobius}
\end{equation}
and the function $\textnormal{Li}_{s}(z)$ is the polylogarithm, a function that appears in many areas of mathematics and physics, which has the series definition
\begin{equation}
\textnormal{Li}_{s}(z)=\sum_{k=1}^{\infty}\frac{z^{k}}{k^{s}} \label{polylog}
\end{equation}
for all $s \in \mathbb{C}$ with $|z|<1$. The identities (\ref{dabigone})-(\ref{dabigone2}) are hence relations between sums over M{\"o}bius function solutions, i.e., integers with either an even or odd number of distinct prime factors, and sums over the non-divisors of the positive integers, where the terms of both sums are weighted with the polylogarithm. These results thus provide a novel connection between the distribution of M{\"o}bius function solutions, non-divisor sums, and the polylogarithms. We prove the results described above in the following sections.

\section{Preliminaries}
\label{prelims}
The prime zeta function $\mathcal{P}\left(s\right)$ is a closely related function to $\zeta(s)$ that has the expressions
\begin{equation}
\mathcal{P}\left(s\right)=\sum_{p}\frac{1}{p^{s}} \label{prime zeta sum}
\end{equation}
and 
\begin{equation}
\mathcal{P}\left(s\right)=\sum_{n=1}^{\infty}\frac{\mu(n)}{n}\log\zeta\left(ns\right).\label{prime zeta ac}
\end{equation}
The expression (\ref{prime zeta ac}) is derived by taking the logarithm of $\zeta(s)$'s Euler product, Taylor expanding, and using the well-known technique of M{\"o}bius inversion. This technique is enabled by the important fact that the sum of the M{\"o}bius function over the divisors of any $n>1$ vanishes. That is,
\begin{equation}
\sum_{m|n}\mu(m)=0 \label{mobius vanishing}
\end{equation}
for all $n>1$. The prime sum (\ref{prime zeta sum}) is absolutely convergent for $\textnormal{Re}(s)>1$ and the M{\"o}bius inversion-derived analytic continuation (\ref{prime zeta ac}) is convergent everywhere in $\textnormal{Re}(s)>0$ except at the points $ns=1$, which arise from the simple pole $\zeta(1)$, and at the points $ns$ where $\zeta\left(ns\right)=0$.

We will generalize (\ref{prime zeta sum})-(\ref{prime zeta ac}) for the consideration of L-functions $L\left(s,\chi\right)$, where $\chi$ is a given Dirichlet character modulo $M$. We may generalize (\ref{prime zeta sum}) by defining the prime L-function with the series 
\begin{equation}
\mathcal{P}_{\chi}\left(s\right)=\sum_{p}\frac{\chi(p)}{p^{s}} \label{primedirseries}
\end{equation}
where $\chi(p)=e^{-i\theta_{p}}$ for primes $p$ such that $p \nmid M$ and $\chi(p)=0$ for primes $p$ such that $p | M$. The case $\chi=M=1$ gives results for the zeta function and is represented by (\ref{prime zeta sum})-(\ref{prime zeta ac}). Below we will consider the case of principal characters $\chi_{0}$, which, by (\ref{principal ch}) have the properties
\begin{equation}
\chi_{0}(p)=\begin{cases} 1 \hspace{.25cm}; \hspace{.25cm} p\nmid M \\ 0 \hspace{.25cm}; \hspace{.25cm} p|M \end{cases}
\end{equation}
and  
\begin{equation}
\chi_{0}^{n}=\chi_{0} \label{obviously}
\end{equation}
for any $n>0$. We next note that one may take the logarithm of $L(s,\chi)$'s Euler product, Taylor expand, and apply M{\"o}bius inversion to give a generalized version of (\ref{prime zeta ac}) relating L-functions $L(s, \chi)$ and their corresponding series $\mathcal{P}_{\chi}(s)$ for $\textnormal{Re}(s)>0$:
\begin{equation}
\mathcal{P}_{\chi}(s)=\sum_{n=1}^{\infty}\frac{\mu(n)}{n}\log L\left(ns,\chi^{n}\right) \label{prime l ac}
\end{equation}
For principal characters $\chi_{0}$, the expressions (\ref{primedirseries}) and (\ref{prime l ac}) have equivalent convergence properties to (\ref{prime zeta sum})-(\ref{prime zeta ac}) because both $\zeta(s)$ and $L\left(s,\chi_{0}\right)$ have a pole at $s=1$ and both have the same nontrivial zeros.

Before proceeding we lastly make a few notes about the polylogarithm (\ref{polylog}). One can apply M{\"o}bius inversion with (\ref{polylog}) to show that, for any $|z|<1$ and $s \in \mathbb{C}$,
\begin{equation}
z=\sum_{n=1}^{\infty}\frac{\mu(n)}{n^{s}}\textnormal{Li}_{s}\left(z^{n}\right). \label{gen x}
\end{equation}
Additionally, a function related to the polylogarithm will repeatedly appear that has the form 
\begin{equation}
\sum_{k=1}^{\infty}\frac{\mathcal{P}_{\chi_{0}}(ks)}{k^{2}}. \label{da function}
\end{equation}
For $\textnormal{Re}(s)>1$ we may expand $\mathcal{P}_{\chi_{0}}(ks)$ using (\ref{primedirseries}), switch the order of summation, and apply (\ref{polylog}) to write 
\begin{equation}
\sum_{k=1}^{\infty}\frac{\mathcal{P}_{\chi_{0}}(ks)}{k^{2}}=\sum_{p \nmid M}\textnormal{Li}_{2}\left(\frac{1}{p^{s}}\right). \label{da function 2}
\end{equation}
We will apply the above functions and their properties below.

\section{Nonlocal correlations in prime-L vs. log-L-functions}
\label{nonlocal corr}

\subsection{The prime-L-function}
We first study the statistical properties of (\ref{prime zeta sum}) and, more generally, (\ref{primedirseries}). We consider $t$ uniformly distributed in $[a,b]$ with $b-a\rightarrow \infty$. It can then be shown, by applying characteristic functions with the fundamental theorem of arithmetic, that the summands of (\ref{primedirseries})'s $\textnormal{Re}\mathcal{P}_{\chi}\left(\sigma+it\right)$ are statistically independent. In general, for $p \neq q$, $c\cos\left(\alpha t\log p+\theta\right)$ and $c'\cos\left(\beta t\log q+\theta'\right)$ are independent random variables for any $c, c', \theta, \theta' \in \mathbb{R}$ and $\alpha,\beta \in \mathbb{Q}$\footnote{See Appendix \ref{app independence} for sketch of proof.}. This is an important fact for the evaluation of expected values, which we will demonstrate first in the proof of the following lemma:
\begin{lemma}
For $t$ uniformly distributed in $[a,b]$ with $b-a\rightarrow \infty$ and $\alpha,\beta \in \mathbb{Q}$,
\begin{align}
E\left\{\textnormal{Re}\mathcal{P}_{\chi}\left(\sigma+i\alpha t\right)\textnormal{Re}\mathcal{P}_{\chi}\left(\sigma'+i\beta t\right)\right\} =\begin{cases} 
       \frac{1}{2}\mathcal{P}_{\chi_{0}}(\sigma+\sigma')\hspace{.1cm}; & \alpha=\beta \\
      0 \hspace{.1cm}; & \alpha \neq \beta
   \end{cases} \label{covar general better}
\end{align}
for all $\sigma+\sigma'>1$ and $\sigma+\sigma'>0$ with $j\left(\sigma+\sigma'\right) \neq 1$ for any $j \in \mathbb{N}$. \label{prime l cov lemma}
\end{lemma}
\begin{proof}
For $t$ uniformly distributed in $[a,b]$, the expected value of a given function $f(.)$ is defined 
$$E\left\{f\left(t\right)\right\}=\frac{1}{b-a}\int_{a}^{b}f\left(x\right)dx.$$
We note from the Fubini and Tonelli theorems that if
\begin{equation}
\sum_{n}\frac{1}{b-a}\int_{a}^{b}\left|f_{n}(t)\right|dt<\infty \label{ft cond}
\end{equation}
for general functions $f_{n}(t)$, then 
\begin{equation}
\frac{1}{b-a}\int_{a}^{b}\sum_{n}f_{n}(t)dt=\sum_{n}\frac{1}{b-a}\int_{a}^{b}f_{n}(t)dt. \label{ft}
\end{equation}
We next apply (\ref{primedirseries}) to note that, for any $\alpha,\beta \in \mathbb{Q}$ and $\sigma,\sigma'>1$,
\begin{align}
\textnormal{Re}\mathcal{P}_{\chi}\left(\sigma+i\alpha t\right)\textnormal{Re}\mathcal{P}_{\chi'}\left(\sigma'+i\beta t\right)=\sum_{p \nmid M,M'}\frac{\cos\left(\alpha t\log p+\theta_{p}\right)\cos\left(\beta t\log p+\theta'_{p}\right)}{p^{\sigma+\sigma'}} \nonumber \\ +2\sum_{p\nmid M}\frac{\cos\left(\alpha t\log p+\theta_{p}\right)}{p^{\sigma}}\sum_{\substack{q<p \\ q \nmid M'}}\frac{\cos\left(\beta t \log q+\theta'_{q}\right)}{q^{\sigma'}}. \label{gen pz multiply}
\end{align}
We then note that the series on the right-hand side of (\ref{gen pz multiply}) satisfy the absolute convergence property
\begin{align}
\sum_{p \nmid M,M'}\frac{\left|\cos\left(\alpha t\log p+\theta_{p}\right)\cos\left(\beta t\log p+\theta'_{p}\right)\right|}{p^{\sigma+\sigma'}} \nonumber \\ +2\sum_{p\nmid M}\frac{\left|\cos\left(\alpha t\log p+\theta_{p}\right)\right|}{p^{\sigma}}\sum_{\substack{q<p \\ q \nmid M'}}\frac{\left|\cos\left(\beta t\log q+\theta'_{q}\right)\right|}{q^{\sigma'}} \nonumber \\ 
\leq \sum_{p}\frac{1}{p^{\sigma+\sigma'}}+2\sum_{p}\frac{1}{p^{\sigma}}\sum_{q<p}\frac{1}{q^{\sigma'}} < \mathcal{P}(\sigma+\sigma')+2\mathcal{P}(\sigma)\mathcal{P}(\sigma'). \label{gen bound}
\end{align}
Note that (\ref{gen bound}) shows that (\ref{gen pz multiply}) satisfies (\ref{ft cond}). Therefore the expected value can be applied term-by-term to (\ref{gen pz multiply}).

We then apply the expected value with $b-a \rightarrow \infty$ and note that, by the statistical independence shown in Appendix \ref{app independence}, the expected value of the double-summation in (\ref{gen pz multiply}) always vanishes because $E\left\{\cos\left(\omega t+\theta\right)\right\}=0$ for all $\omega, \theta \in \mathbb{R}$. Meanwhile the first summation vanishes unless $\alpha=\beta$ since $E\left\{\cos\left(\omega t+\theta\right)\cos\left(\omega' t+\theta'\right)\right\}=0$ for all $\omega \neq \omega'$ and all $\theta, \theta' \in \mathbb{R}$. This gives
\begin{align}
E\left\{\textnormal{Re}\mathcal{P}_{\chi}\left(\sigma+i\alpha t\right)\textnormal{Re}\mathcal{P}_{\chi'}\left(\sigma'+i\beta t\right)\right\} =\begin{cases} 
       \frac{1}{2}\sum_{p \nmid M, M'}\frac{\cos\left(\theta_{p}-\theta_{p}'\right)}{p^{\sigma+\sigma'}}\hspace{.1cm}; & \alpha=\beta \\
      0 \hspace{.1cm}; & \alpha \neq \beta
   \end{cases} \label{covar general}
\end{align}
For $\chi=\chi'$, (\ref{covar general}) is given by (\ref{covar general better}). Then considering the domains of convergence of (\ref{primedirseries}) and (\ref{prime l ac}) completes the proof.
\end{proof}
The result (\ref{covar general better}) shows that $\mathcal{P}_{\chi}(s)$ has no nonlocal correlation structure involving the rational numbers $\alpha$ and $\beta$. For example, (\ref{covar general better}) immediately implies
\begin{align}
E\left\{\textnormal{Re}\mathcal{P}_{\chi}\left(\alpha\left(\sigma+it\right)\right)\textnormal{Re}\mathcal{P}_{\chi}\left(\beta\left(\sigma+it\right)\right)\right\} =\begin{cases} 
       \frac{1}{2}\textnormal{Re}\mathcal{P}_{\chi_{0}}\left(2\alpha\sigma\right) \hspace{.1cm}; & \alpha=\beta \\
      0 \hspace{.1cm}; & \alpha \neq \beta
   \end{cases} \label{covar}
\end{align}
for all $2\alpha\sigma>1$ and $\sigma>0$ with $2\alpha j \sigma \neq 1$ for any $j \in \mathbb{N}$, as well as
\begin{align}
E\left\{\textnormal{Re}\mathcal{P}_{\chi}\left(\sigma+i\alpha t\right)\textnormal{Re}\mathcal{P}_{\chi}\left(\sigma+i\beta t\right)\right\} =\begin{cases} 
       \frac{1}{2}\textnormal{Re}\mathcal{P}_{\chi_{0}}\left(2\sigma\right) \hspace{.1cm}; & \alpha=\beta \\
      0 \hspace{.1cm}; & \alpha \neq \beta
  \end{cases} \label{covar vert}
\end{align}
for all $2\sigma>1$ and $\sigma>0$ with $2j \sigma \neq 1$ for any $j \in \mathbb{N}$. The results (\ref{covar}) and (\ref{covar vert}) show a lack of  \textit{nonlocal diagonal correlation} and \textit{nonlocal vertical correlation} respectively, regardless of any multiplicative relationship between $\alpha$ and $\beta$. We will show this to be strongly contrasting with principal L-functions. To do so we will directly apply the result (\ref{covar general better}).

\subsection{The log-L-function}
Consider a principal Dirichlet character $\chi_{0}$ modulo $M$, its corresponding L-function $L\left(s,\chi_{0}\right)$, and its prime L-function $\mathcal{P}_{\chi_{0}}(s)$. In this section we will show that $\log\left|L\left(s,\chi_{0}\right)\right|$, in contrast to $\textnormal{Re}\mathcal{P}_{\chi_{0}}(s)$, displays significant nonlocal correlations. 

\begin{thm}
For $t$ uniformly distributed in $[a,b]$ with $b-a\rightarrow \infty$, $j \in \mathbb{N}$, $\alpha,\beta \in \mathbb{Q}$, and $\alpha \geq \beta$,
\begin{align}
E\left\{\log\left|L\left(\sigma+i\alpha t,\chi_{0}\right)\right|\log\left|L\left(\sigma'+i\beta t,\chi_{0}\right)\right|\right\} \nonumber \\= \begin{cases} 
       \frac{\beta}{2\alpha}\sum_{k=1}^{\infty}\frac{\mathcal{P}_{\chi_{0}}\left(k\left(\sigma+\frac{\alpha}{\beta}\sigma'\right)\right)}{k^{2}} \hspace{.1cm}; & \beta|\alpha, \hspace{.42cm} \sigma+\frac{\alpha}{\beta}\sigma'>0, \hspace{.1cm}\sigma+\frac{\alpha}{\beta}\sigma' \neq \frac{1}{j}\\
      \frac{1}{2\alpha \beta}\sum_{k=1}^{\infty}\frac{\mathcal{P}_{\chi_{0}}\left(k\left(\beta\sigma+\alpha \sigma'\right)\right)}{k^{2}} \hspace{.1cm}; & \beta \nmid \alpha, \hspace{.25cm} \beta\sigma+\alpha \sigma'>0, \hspace{.1cm}\beta \sigma+\alpha \sigma' \neq \frac{1}{j}
   \end{cases} \label{logzeta cov}
\end{align} 
and
\begin{align}
E\left\{\log\left|L\left(\sigma+i \alpha t,\chi_{0}\right)\right|\log\left|L\left(\sigma'+i\beta t,\chi_{0}\right)\right|\right\}\nonumber \\  = \begin{cases} 
       \frac{\beta}{2\alpha}\sum_{p \nmid M}\textnormal{Li}_{2}\left(\frac{1}{p^{\sigma+\frac{\alpha}{\beta}\sigma'}}\right)  \hspace{.1cm}; & \beta|\alpha , \hspace{.42cm}\sigma+\frac{\alpha}{\beta}\sigma'>1\\
      \frac{1}{2\alpha \beta}\sum_{p \nmid M}\textnormal{Li}_{2}\left(\frac{1}{p^{\beta \sigma+\alpha \sigma'}}\right) \hspace{.1cm}; & \beta \nmid \alpha, \hspace{.25cm} \beta \sigma+\alpha \sigma'>1
   \end{cases} \label{logzeta cov 2}
\end{align}
\label{logzeta cov thm}
\end{thm}
\begin{proof}
We begin by considering $\alpha \geq \beta$ and $\beta \sigma>1$ so that we may apply the Euler product and Taylor expansion for $\log \left|L\left(s,\chi_{0}\right)\right|$ with (\ref{obviously}) to write
\begin{align}
\log\left|L\left(\sigma+i \alpha t,\chi_{0}\right)\right|\log\left|L\left(\sigma'+i\beta t,\chi_{0}\right)\right| \nonumber \\ =\sum_{k=1}^{\infty}\frac{\textnormal{Re}\mathcal{P}_{\chi_{0}}\left(k\left(\sigma+i\alpha t\right)\right)\textnormal{Re}\mathcal{P}_{\chi_{0}}\left(k\left(\sigma'+i\beta t\right)\right)}{k^{2}} \nonumber \\
+2\sum_{k=1}^{\infty}\sum_{k'<k}\frac{\textnormal{Re}\mathcal{P}_{\chi_{0}}\left(k\left(\sigma+i\alpha t\right)\right)\textnormal{Re}\mathcal{P}_{\chi_{0}}\left(k'\left(\sigma'+i\beta t\right)\right)}{kk'}  \label{app error 3}
\end{align}
We next will apply the expectation to (\ref{app error 3}) under $t$ uniformly distributed in $[a,b]$ with $b-a\rightarrow \infty$. We first note from (\ref{app error 3}) and (\ref{gen bound}) that (\ref{app error 3})'s first summation has the absolute upper bound
\begin{equation}
\sum_{k=1}^{\infty}\frac{\mathcal{P}_{\chi_{0}}\left(k\sigma\right)\mathcal{P}_{\chi_{0}}\left(k \sigma'\right)}{k^{2}}<\infty \label{term 1 upper bound}
\end{equation}
and that (\ref{app error 3})'s second summation has the absolute upper bound 
\begin{align}
2\sum_{k=1}^{\infty}\frac{\mathcal{P}_{\chi_{0}}\left(k \sigma\right)}{k}\sum_{k'<k}\frac{\mathcal{P}_{\chi_{0}}\left(k'\sigma'\right)}{k'}
< 2\mathcal{P}_{\chi_{0}}\left(\sigma'\right)\sum_{k=1}^{\infty}\mathcal{P}_{\chi_{0}}\left(k \sigma\right)<\infty, \label{term 2 upper bound}
\end{align}
the finiteness of which may be easily shown from the asymptotic decay rates implied by (\ref{primedirseries}) and the ratio test. The finite bounds (\ref{term 1 upper bound}) and (\ref{term 2 upper bound}) show that (\ref{app error 3}) satisfies the condition (\ref{ft cond}). We may therefore apply the expectation term-by-term. 

We first consider the case $\alpha=\beta$ and apply the expectation, using (\ref{covar general better}) to show that then only (\ref{app error 3})'s first summation has nonzero expected value, which gives the result
\begin{equation}
E\left\{\log\left|L\left(\sigma+i\alpha t,\chi_{0}\right)\right|\log\left|L\left(\sigma'+i\alpha t,\chi_{0}\right)\right|\right\}=\frac{1}{2}\sum_{k=1}^{\infty}\frac{\mathcal{P}_{\chi_{0}}\left(k(\sigma+\sigma')\right)}{k^{2}}. \label{logzeta var}
\end{equation}
We next consider the case $\alpha \neq \beta$ and apply the expectation, using (\ref{covar general better}) to show that then the only terms with nonzero expected value are those in (\ref{app error 3})'s second summation where $k\alpha=k'\beta$, i.e., terms where $k'=k\alpha/\beta \in \mathbb{N}$. This constraint on $k'$ implies that $\beta|k\alpha$, i.e., either $\beta|\alpha$ or $\beta\nmid \alpha$ and $\beta |k$. Considering these two cases with (\ref{app error 3}), simplifying, and comparing to (\ref{logzeta var}) gives the following result for $\alpha \geq \beta$:
\begin{align}
E\left\{\log\left|L\left(\sigma+i\alpha t,\chi_{0}\right)\right|\log\left|L\left(\sigma'+i\beta t,\chi_{0}\right)\right|\right\} \nonumber \\= \begin{cases} 
       \frac{\beta}{2\alpha}\sum_{k=1}^{\infty}\frac{\mathcal{P}_{\chi_{0}}\left(k\left(\sigma+\frac{\alpha}{\beta}\sigma'\right)\right)}{k^{2}} \hspace{.1cm}; & \beta|\alpha \\
      \frac{1}{2\alpha \beta}\sum_{l=1}^{\infty}\frac{\mathcal{P}_{\chi_{0}}\left(l\left(\beta \sigma+\alpha \sigma'\right)\right)}{l^{2}}\hspace{.1cm}; & \beta \nmid \alpha
   \end{cases}  \label{logzeta cov raw}
\end{align}
where the lower-right result is achieved by simply noting $\beta|k$ implies $k=l\beta$ for some $l \in \mathbb{N}$. 

We next note that (\ref{logzeta cov raw}) may be written using the analytic continuation (\ref{prime l ac}), which extends (\ref{logzeta cov raw})'s domain of definition to $\sigma+\frac{\alpha}{\beta}\sigma'>0$ with $kj\left(\sigma+\frac{\alpha}{\beta}\sigma'\right) \neq 1$ for any $k,j \in \mathbb{N}$ in the case of $\beta|\alpha$ and $kj \left(\beta \sigma+\alpha \sigma'\right) \neq 1$ for any $k,j \in \mathbb{N}$ in the case of $\beta \nmid \alpha$. This proves (\ref{logzeta cov}) and its constraints on $\sigma$. We then note that (\ref{logzeta cov raw}) may be written using the prime sum (\ref{primedirseries}), giving
\begin{align}
E\left\{\log\left|L\left(\sigma+i\alpha t,\chi_{0}\right)\right|\log\left|L\left(\sigma'+i\beta t,\chi_{0}\right)\right|\right\} \nonumber \\= \begin{cases} 
       \frac{\beta}{2\alpha}\sum_{k=1}^{\infty}\frac{1}{k^{2}}\sum_{p \nmid M}\frac{1}{p^{k\left(\sigma+\frac{\alpha}{\beta}\sigma'\right)}} \hspace{.1cm}; & \beta|\alpha \\
      \frac{1}{2\alpha \beta}\sum_{l=1}^{\infty}\frac{1}{l^{2}}\sum_{p \nmid M}\frac{1}{p^{l\left(\beta \sigma+\alpha \sigma'\right)}}\hspace{.1cm}; & \beta \nmid \alpha
   \end{cases}  \label{logzeta cov raw sum}
\end{align}
The domain of convergence for (\ref{primedirseries}) implies that (\ref{logzeta cov raw sum}) is convergent for $\sigma+\frac{\alpha}{\beta}\sigma'>1$ in the case of $\beta |\alpha$ and $\beta \sigma+\alpha \sigma'>1$ in the case of $\beta \nmid \alpha$. We then lastly note that, by (\ref{da function 2}), (\ref{logzeta cov raw}) is equivalent to (\ref{logzeta cov 2}). 
\end{proof}

\subsubsection{Nonlocal diagonal correlation}
\label{integer diagonal correlation}
We first apply Theorem \ref{logzeta cov thm} to describe $\log\left|L\left(s,\chi_{0}\right)\right|$'s nonlocal diagonal covariance function, which we denote by $R_{\chi_{0}}^{\textnormal{diag}}(\alpha,\beta,\sigma)$, and which, by (\ref{logzeta cov})-(\ref{logzeta cov 2}), may be written
\begin{align}
R_{\chi_{0}}^{\textnormal{diag}}(\alpha,\beta,\sigma)=E\left\{\log\left|L\left(\alpha \left(\sigma+it\right),\chi_{0}\right)\right|\log\left|L\left(\beta \left(\sigma+it\right),\chi_{0}\right)\right|\right\}\nonumber \\=
\begin{cases} 
       \frac{\beta}{2\alpha}\sum_{k=1}^{\infty}\frac{\mathcal{P}_{\chi_{0}}(2k\alpha\sigma)}{k^{2}} \hspace{.1cm}; & \beta|\alpha, \hspace{.42cm} \sigma>0, \hspace{.1cm}\sigma \neq \frac{1}{2j\alpha}, j \in \mathbb{N}\\
      \frac{1}{2\alpha \beta}\sum_{k=1}^{\infty}\frac{\mathcal{P}_{\chi_{0}}(2k\alpha \beta\sigma)}{k^{2}} \hspace{.1cm}; & \beta \nmid \alpha, \hspace{.25cm} \sigma>0, \hspace{.1cm}\sigma \neq \frac{1}{2j\alpha \beta}, j \in \mathbb{N}
   \end{cases}\nonumber \\ =
\begin{cases} 
       \frac{\beta}{2\alpha}\sum_{p \nmid M}\textnormal{Li}_{2}\left(\frac{1}{p^{2\alpha \sigma}}\right)  \hspace{.1cm}; & \beta|\alpha , \hspace{.42cm} \sigma>\frac{1}{2\alpha}\\
      \frac{1}{2\alpha \beta}\sum_{p \nmid M}\textnormal{Li}_{2}\left(\frac{1}{p^{2\alpha \beta \sigma}}\right) \hspace{.1cm}; & \beta \nmid \alpha, \hspace{.25cm} \sigma>\frac{1}{2\alpha \beta} 
   \end{cases} \label{diag covar}
\end{align}
It is clear that (\ref{diag covar}) has much richer structure than $\textnormal{Re}\mathcal{P}_{\chi_{0}}(s)$'s corresponding diagonal covariance function  (\ref{covar}). From (\ref{diag covar}) we see that $R_{\chi_{0}}^{\textnormal{diag}}(\alpha,\beta,\sigma)$ has the noteworthy and useful property 
\begin{equation}
R_{\chi_{0}}^{\textnormal{diag}}(\alpha ,\beta,\sigma)= \begin{cases} 
       \frac{\beta}{\alpha}R_{\chi_{0}}^{\textnormal{diag}}(\alpha,\alpha,\sigma) \hspace{.1cm}; & \beta|\alpha \\
      \frac{1}{\alpha \beta}R_{\chi_{0}}^{\textnormal{diag}}(\alpha \beta,\alpha \beta,\sigma) \hspace{.1cm}; & \beta \nmid \alpha
   \end{cases}  \label{logzeta cov alternate}
\end{equation}
To better study the correlations implied by (\ref{diag covar})-(\ref{logzeta cov alternate}), we consider the particular covariance function $R_{1}^{\textnormal{diag}}(\alpha,\beta,\sigma)$, i.e., the covariance between $\log\left|\zeta\left(\alpha \left(\sigma+it\right)\right)\right|$ and $\log\left|\zeta\left(\beta \left(\sigma+it\right)\right)\right|$, but the general conclusions are very similar for other principal L-functions. We define the correlation as
\begin{equation}
\rho_{1}^{\textnormal{diag}}(\alpha,\beta,\sigma)=\frac{R_{1}^{\textnormal{diag}}(\alpha,\beta,\sigma)}{\sqrt{R_{1}^{\textnormal{diag}}(\alpha,\alpha,\sigma)R_{1}^{\textnormal{diag}}(\beta,\beta,\sigma)}} \label{corr}
\end{equation}
for $\alpha \geq \beta$ and as $\rho_{1}^{\textnormal{diag}}(\alpha,\beta,\sigma)=\rho_{1}^{\textnormal{diag}}(\beta,\alpha,\sigma)$ for $\alpha<\beta$. We can numerically approximate (\ref{diag covar}), (\ref{logzeta cov alternate}), and (\ref{corr}) to produce correlation tables like those depicted in Figure \ref{cor table fig} for $\alpha=n$ and  $\beta=m$ with $n,m \in \mathbb{N}$. It is clear from these tables that there is much higher correlation between $\log\left|\zeta\left(n\left(\sigma+it\right)\right)\right|$ and $\log\left|\zeta\left(m\left(\sigma+it\right)\right)\right|$ when either $n|m$ or $m|n$. We may explain this by first noting from (\ref{prime zeta sum}), (\ref{diag covar}), and (\ref{logzeta cov alternate}) that, as $\alpha \rightarrow \infty$,
\begin{equation}
R_{1}^{\textnormal{diag}}(\alpha,\beta,\sigma)\sim \begin{cases} 
       \frac{\beta}{\alpha}\frac{1}{2^{2\alpha \sigma+1}} \hspace{.1cm}; & \beta|\alpha \\
      \frac{1}{\alpha \beta}\frac{1}{2^{2\alpha \beta \sigma+1}} \hspace{.1cm}; & \beta \nmid \alpha
   \end{cases}  \label{logzeta cov asy}
\end{equation}
then applying (\ref{logzeta cov asy}) in (\ref{corr}) and simplifying shows that, as $\alpha \rightarrow \infty$,
\begin{equation}
\rho_{1}^{\textnormal{diag}}(\alpha,\beta,\sigma)\sim \begin{cases} 
       \frac{\beta}{\alpha}2^{(\beta-\alpha)\sigma} \hspace{.1cm}; & \beta|\alpha \\
      \frac{1}{\alpha \beta}2^{(\beta+\alpha-2\alpha \beta)\sigma} \hspace{.1cm}; & \beta \nmid \alpha
   \end{cases}\hspace{.25cm}
=\begin{cases} 
       O\left(2^{-\alpha \sigma}\right) \hspace{.1cm}; & \beta|\alpha \\
      O\left(2^{-\left(2\beta-1\right)\alpha \sigma}\right) \hspace{.1cm}; & \beta \nmid \alpha
   \end{cases}
\label{logzeta corr asy}
\end{equation}
This result shows that, for fixed $\beta$, the correlation decays at a rate $O\left(2^{-\alpha \sigma}\right)$ for $\alpha$ such that $\beta|\alpha$, but decays at a much faster rate $O\left(2^{-(2\beta-1)\alpha \sigma}\right)$ for $\alpha$ such that $\beta \nmid \alpha$. These decay rates are clearly visible in Figure \ref{cor plot fig}. These results thus show that there is potentially high correlation between $\log\left|\zeta\left(\alpha\left(\sigma+it\right)\right)\right|$ and $\log\left|\zeta\left(\beta\left(\sigma+it\right)\right)\right|$, depending on whether $\beta|\alpha$ or $\beta \nmid \alpha$. As mentioned above, the same reasoning gives extremely similar results for other principal L-functions, with $2$ in (\ref{logzeta cov asy})-(\ref{logzeta corr asy}) replaced by the minimum prime $p$ such that $p \nmid M$. This gives the general result (\ref{nonloc diag corr 1}). 

\begin{figure}
\centering
\includegraphics[width=\linewidth]{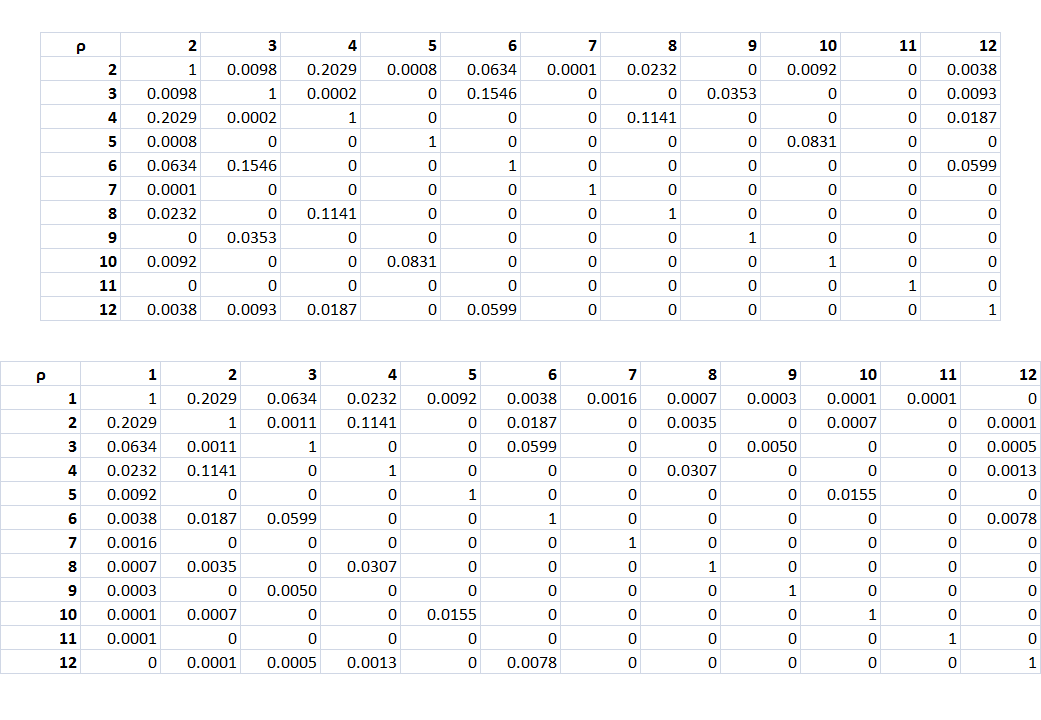}
\caption{Correlation tables produced using (\ref{diag covar}), (\ref{logzeta cov alternate}), and (\ref{corr}) to approximate $\rho_{1}^{\textnormal{diag}}(n,m,\sigma)$ to four decimal places with $2 \leq n,m \leq 12 \in \mathbb{N}$ and $\sigma=1/2$ (Top) and $1\leq n,m \leq 12 \in \mathbb{N}$ and $\sigma=1$ (Bottom).}
\label{cor table fig}
\end{figure}

\begin{figure}
\centering
\includegraphics[width=\linewidth]{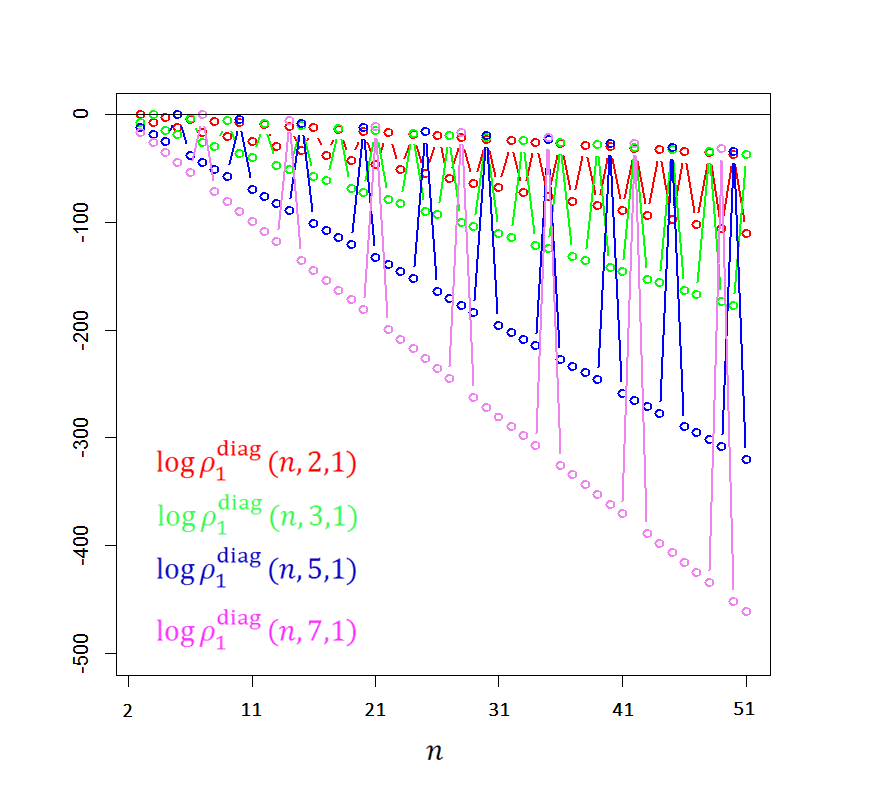}
\caption{Plots of $\log \rho_{1}^{\textnormal{diag}}(n,m,1)$ approximated using (\ref{diag covar}), (\ref{logzeta cov alternate}), (\ref{corr}) with $m=2, 3, 5, 7$ and $2\leq n \leq 51 \in \mathbb{N}$.}
\label{cor plot fig}
\end{figure}

\subsubsection{Nonlocal vertical correlation}
\label{integer vertical correlation}
We next apply Theorem \ref{logzeta cov thm} to describe $\log\left|L\left(s,\chi_{0}\right)\right|$'s nonlocal vertical covariance function, which we denote by $R_{\chi_{0}}^{\textnormal{vert}}(\alpha,\beta,\sigma)$, and which, by (\ref{logzeta cov})-(\ref{logzeta cov 2}), may be written
\begin{align}
R_{\chi_{0}}^{\textnormal{vert}}(\alpha,\beta,\sigma)=E\left\{\log\left|L\left(\sigma+i\alpha t,\chi_{0}\right)\right|\log\left|L\left(\sigma+i\beta t,\chi_{0}\right)\right|\right\}\nonumber \\=\begin{cases} 
       \frac{\beta}{2\alpha}\sum_{k=1}^{\infty}\frac{\mathcal{P}_{\chi_{0}}\left(k\left(1+\frac{\alpha}{\beta}\right)\sigma\right)}{k^{2}} \hspace{.1cm}; & \beta|\alpha, \hspace{.42cm} \sigma>0, \hspace{.1cm}\sigma \neq \frac{\beta}{j(\beta+\alpha)}, j \in \mathbb{N}\\
      \frac{1}{2\alpha \beta}\sum_{k=1}^{\infty}\frac{\mathcal{P}_{\chi_{0}}\left(k\left(\beta+\alpha\right)\sigma\right)}{k^{2}} \hspace{.1cm}; & \beta \nmid \alpha, \hspace{.25cm} \sigma>0, \hspace{.1cm}\sigma \neq \frac{1}{j(\beta+\alpha)}, j \in \mathbb{N}
   \end{cases} \nonumber \\=
 \begin{cases} 
       \frac{\beta}{2\alpha}\sum_{p \nmid M}\textnormal{Li}_{2}\left(\frac{1}{p^{\left(1+\frac{\alpha}{\beta}\right)\sigma}}\right)  \hspace{.1cm}; & \beta|\alpha, \hspace{.42cm} \sigma>\frac{\beta}{\beta+\alpha} \\
      \frac{1}{2\alpha \beta}\sum_{p \nmid M}\textnormal{Li}_{2}\left(\frac{1}{p^{(\beta+\alpha)\sigma}}\right) \hspace{.1cm}; & \beta \nmid \alpha, \hspace{.25cm} \sigma>\frac{1}{\beta+\alpha}
   \end{cases}\label{logzeta vert cov}
\end{align}
It is again clear that (\ref{diag covar}) strongly contrasts with $\textnormal{Re}\mathcal{P}_{\chi_{0}}(s)$'s corresponding vertical covariance function (\ref{covar vert}). Similarly to Section \ref{integer diagonal correlation}, to study the correlations implied by (\ref{logzeta vert cov}), we consider $R_{1}^{\textnormal{vert}}(\alpha,\beta,\sigma)$, i.e., the covariance between $\log\left|\zeta\left(\sigma+i\alpha t\right)\right|$ and $\log\left|\zeta\left(\sigma+i\beta t\right)\right|$. We define the correlation similarly to (\ref{corr}) as 
\begin{align}
\rho_{1}^{\textnormal{vert}}(\alpha,\beta,\sigma)
=\frac{R_{1}^{\textnormal{vert}}(\alpha,\beta,\sigma)}{\sqrt{R_{1}^{\textnormal{vert}}(\alpha,\alpha,\sigma)R_{1}^{\textnormal{vert}}(\beta,\beta,\sigma)}} 
=\frac{R_{1}^{\textnormal{vert}}(\alpha,\beta,\sigma)}{R_{1}^{\textnormal{vert}}(1,1,\sigma)}\label{vert cor}
\end{align} 
for $\alpha \geq \beta$ and as $\rho_{1}^{\textnormal{vert}}(\alpha,\beta,\sigma)=\rho_{1}^{\textnormal{vert}}(\beta,\alpha,\sigma)$ for $\alpha<\beta$. We can numerically approximate (\ref{logzeta vert cov})-(\ref{vert cor}) to produce a correlation table like that depicted in Figure \ref{vert cor table fig} for $\alpha=n$ and $\beta=m$ with $n,m \in \mathbb{N}$. This table shows much higher correlation between $\log\left|\zeta\left(\sigma+int\right)\right|$ and $\log\left|\zeta\left(\sigma+im t\right)\right|$ when either $n|m$ or $m|n$. We can explain this similarly to Section \ref{integer diagonal correlation} by first noting from (\ref{prime zeta sum}) and (\ref{logzeta vert cov}) that, as $\alpha \rightarrow \infty$,
\begin{equation}
R_{1}^{\textnormal{vert}}(\alpha,\beta,\sigma)\sim \begin{cases} 
       \frac{\beta}{\alpha}\frac{1}{2^{\left(1+\frac{\alpha}{\beta}\right)\sigma+1}} \hspace{.1cm}; & \beta|\alpha \\
      \frac{1}{\alpha \beta}\frac{1}{2^{\left(\beta+\alpha\right)\sigma+1}} \hspace{.1cm}; & \beta \nmid \alpha
   \end{cases}  \label{logzeta vert cov asy}
\end{equation}
which we may apply in (\ref{logzeta vert cov}) and simplify to show that, as $\alpha \rightarrow \infty$,
\begin{equation}
\rho_{1}^{\textnormal{vert}}(\alpha,\beta,\sigma)\sim\begin{cases} 
     \frac{\beta}{\alpha}2^{\left(1-\frac{\alpha}{\beta}\right)\sigma} \hspace{.1cm}; & \beta|\alpha \\
     \frac{1}{\alpha \beta}2^{\left(2-\beta-\alpha\right)\sigma} \hspace{.1cm}; & \beta \nmid \alpha
   \end{cases} \hspace{.25cm}
=\begin{cases} 
      O\left(2^{-\frac{\alpha}{\beta}\sigma}\right) \hspace{.1cm}; & \beta|\alpha \\
      O\left(2^{-\alpha\sigma}\right) \hspace{.1cm}; & \beta \nmid \alpha
   \end{cases}  \label{vert cor asy}
\end{equation}
The decay rates in (\ref{vert cor asy}) are visible in Figure \ref{vert cor plot fig}. These results show that, along with the nonlocal diagonal correlation described in Section \ref{integer diagonal correlation}, there is potentially high correlation between $\log\left|\zeta\left(\sigma+i\alpha t\right)\right|$ and $\log\left|\zeta\left(\sigma+i\beta t\right)\right|$, depending on whether $\beta|\alpha$ or $\beta \nmid \alpha$. Additionally, as in Section \ref{integer diagonal correlation}, very similar results hold for other principal L-functions, with $2$ in (\ref{logzeta vert cov asy})-(\ref{vert cor asy}) replaced by the minimum prime $p$ such that $p \nmid M$, giving the general result (\ref{nonloc vert corr 1}). For both the diagonal correlations described in Section \ref{integer diagonal correlation} as well as the vertical correlations described in this section, we refer to correlations in the case $\beta|\alpha$ as \textit{resonant} correlations and those in the case $\beta \nmid \alpha$ as \textit{dissonant} correlations.

\begin{figure}
\centering
\includegraphics[width=\linewidth]{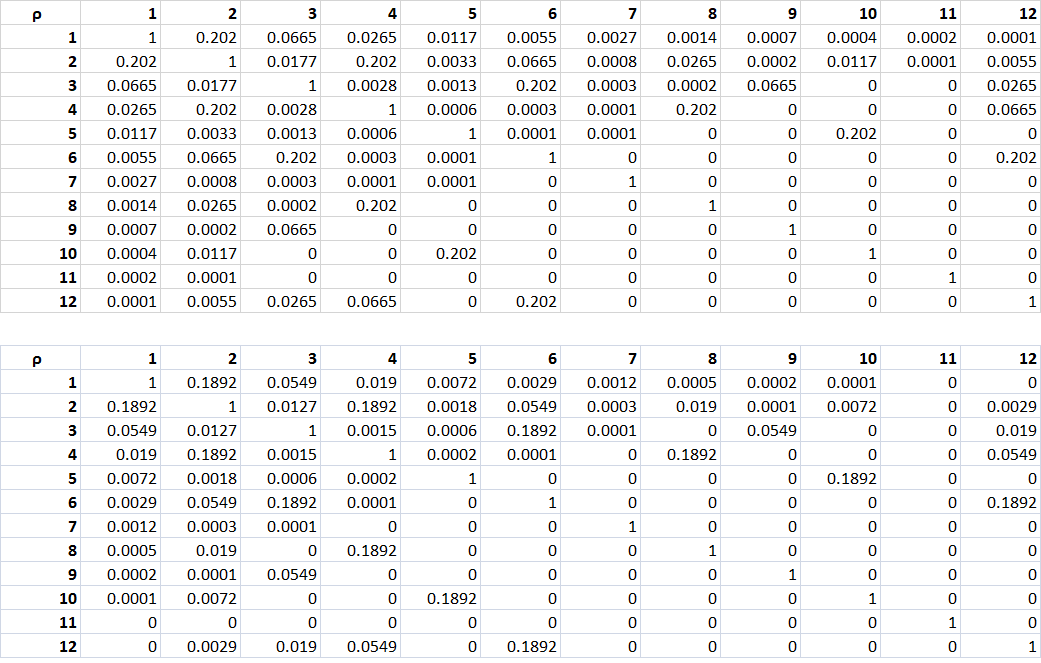}
\caption{Correlation tables produced using (\ref{logzeta vert cov}) and (\ref{vert cor}) to approximate $\rho_{1}^{\textnormal{vert}}(n,m,\sigma)$ to four decimal places with $1 \leq n,m \leq 12 \in \mathbb{N}$ and $\sigma=3/4$ (Top), $\sigma=1$ (Bottom).}
\label{vert cor table fig}
\end{figure}

\begin{figure}
\centering
\includegraphics[width=\linewidth]{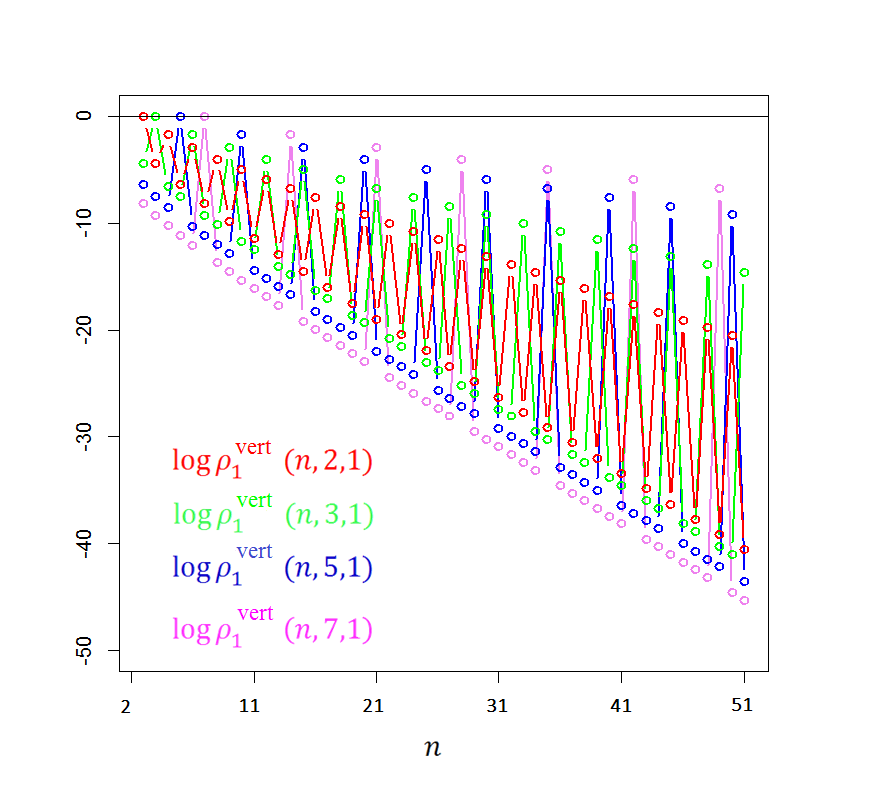}
\caption{Plots of $\log \rho_{1}^{\textnormal{vert}}(n,m,1)$ approximated using (\ref{logzeta vert cov}) and (\ref{vert cor}) with $m=2, 3, 5, 7$ and $2\leq n \leq 51 \in \mathbb{N}$.}
\label{vert cor plot fig}
\end{figure}

\section{M{\"o}bius function solutions, non-divisor sums, and polylogarithms}
In this section we will leverage the above results to prove Theorem \ref{bigthm1}. We will first prove a lemma in the following subsection, which describes the statistical behavior of gaps between the prime-L-function $\textnormal{Re}\mathcal{P}_{\chi_{0}}(s)$ and the log-L-function $\log\left|L\left(s,\chi_{0}\right)\right|$.

\subsection{Gaps between prime-L-functions and log-L-functions}
We have seen that while $\log\left|L\left(s,\chi_{0}\right)\right|$ displays nonlocal diagonal correlation, $\textnormal{Re}\mathcal{P}_{\chi_{0}}(s)$ does not. We will apply this fact to study the statistical behavior of the difference between $\textnormal{Re}\mathcal{P}_{\chi_{0}}(s)$ and $\log\left|L\left(s,\chi_{0}\right)\right|$, i.e., 
\begin{equation}
\varepsilon_{\chi_{0}}\left(\sigma,t\right)=\textnormal{Re}\mathcal{P}_{\chi_{0}}\left(\sigma+it\right)-\log\left|L\left(\sigma+it,\chi_{0}\right)\right|.
\label{error}
\end{equation}
The difference (\ref{error}) contains the contributions to $\log\left|L\left(\sigma+it,\chi_{0}\right)\right|$ that give rise to the nonlocal correlations described in Section \ref{nonlocal corr}. Along with the definition (\ref{error}), we will additionally consider the representation 
\begin{equation}
\varepsilon_{\chi_{0}}\left(\sigma,t\right)=\sum_{n=2}^{\infty}\frac{\mu(n)}{n}\log\left|L\left(n\left(\sigma+it\right),\chi_{0}\right)\right|,
\label{error 2}
\end{equation}
which immediately follows from the analytic continuation (\ref{prime l ac}) and (\ref{obviously}). We will apply (\ref{error})-(\ref{error 2}) to produce two different representations for the mean square of $\varepsilon_{\chi_{0}}\left(\sigma,t \right)$ using Theorem \ref{logzeta cov thm}'s covariance function result.

\begin{lemma}
For $t$ uniformly distributed in $[a,b]$ with $b-a\rightarrow \infty$ and all $\sigma>1/4$,
\begin{equation}
E\left\{\varepsilon^{2}_{\chi_{0}}\left(\sigma,t\right)\right\}=-\sum_{n=2}^{\infty}\frac{\mu(n)}{n^{2}}R_{\chi_{0}}^{\textnormal{diag}}(n,n,\sigma) \label{simp mse}
\end{equation}
and
\begin{align}
E\left\{\varepsilon^{2}_{\chi_{0}}\left(\sigma,t\right)\right\}=\nonumber \\ \sum_{n=2}^{\infty}\frac{\mu(n)}{n^{2}}\left(-\left(\mu(n)+2\right)R_{\chi_{0}}^{\textnormal{diag}}(n,n,\sigma)+2\sum_{\substack{m \nmid n\\ m<n}}\frac{\mu(m)}{m^{2}}R_{\chi_{0}}^{\textnormal{diag}}(nm,nm,\sigma)\right) \label{mse still still less raw}
\end{align} \label{mse thm}
\end{lemma}

\begin{proof}
We first note from (\ref{covar}) that 
\begin{equation}
E\left\{\left(\textnormal{Re}\mathcal{P}_{\chi_{0}}\left(\sigma+it\right)\right)^{2}\right\}=\frac{\mathcal{P}_{\chi_{0}}(2\sigma)}{2}. \label{pzeta var}
\end{equation}
We then square (\ref{error}) and apply the expectation, noting (\ref{pzeta var}) as well as (\ref{logzeta cov}) with $n=m=1$ to write 
\begin{align}
E\left\{\varepsilon_{\chi_{0}}^{2}\left(\sigma,t\right)\right\}=\frac{\mathcal{P}_{\chi_{0}}(2\sigma)}{2}+\frac{1}{2}\sum_{k=1}^{\infty}\frac{\mathcal{P}_{\chi_{0}}(2k\sigma)}{k^{2}} \nonumber \\ -2E\left\{\textnormal{Re}\mathcal{P}_{\chi_{0}}\left(\sigma+it\right)\log\left|L\left(\sigma+it,\chi_{0}\right)\right|\right\}. \label{simp step 1}
\end{align}
We consider $\sigma>1$ so we may apply the Euler product and Taylor expansion for $\log\left|L\left(s,\chi_{0}\right)\right|$ with (\ref{obviously}) to write 
\begin{equation}
\textnormal{Re}\mathcal{P}_{\chi_{0}}\left(\sigma+it\right)\log\left|L\left(\sigma+it,\chi_{0}\right)\right|=\textnormal{Re}\mathcal{P}_{\chi_{0}}\left(\sigma+it\right)\sum_{k=1}^{\infty}\frac{\textnormal{Re}\mathcal{P}_{\chi_{0}}\left(k\left(\sigma+it\right)\right)}{k}. \label{simp pre step 1}
\end{equation}
We then note that (\ref{simp pre step 1}) has the absolute upper bound
\begin{equation}
\mathcal{P}(\sigma)\sum_{k=1}^{\infty}\frac{\mathcal{P}(k\sigma)}{k}<\infty.
\label{simp bound}
\end{equation}
The finite bound (\ref{simp bound}) shows that (\ref{simp pre step 1}) satisfies the condition (\ref{ft cond}). We may therefore apply the expectation to (\ref{simp pre step 1}) term-by-term. We then note from (\ref{covar}) that the only term with nonvanishing expected value is the term corresponding to $k=1$, which gives
\begin{equation}
E\left\{\textnormal{Re}\mathcal{P}_{\chi_{0}}\left(\sigma+it\right)\log\left|L\left(\sigma+it,\chi_{0}\right)\right|\right\}=\frac{\mathcal{P}_{\chi_{0}}(2\sigma)}{2}.  \label{simp step 3}
\end{equation}
Applying (\ref{simp step 3}) in (\ref{simp step 1}) and cancelling terms then gives 
\begin{equation}
E\left\{\varepsilon^{2}_{\chi_{0}}\left(\sigma,t\right)\right\}=\frac{1}{2}\sum_{k=2}^{\infty}\frac{\mathcal{P}_{\chi_{0}}(2k\sigma)}{k^{2}}. \label{simp step 4}
\end{equation}
We then use M{\"o}bius inversion to show that, for any $v \in \mathbb{R}$,
\begin{equation}
\mathcal{P}_{\chi_{0}}(2\sigma)=\sum_{n=1}^{\infty}\frac{\mu(n)}{n^{v}}\sum_{k=1}^{\infty}\frac{\mathcal{P}_{\chi_{0}}(2kn\sigma)}{k^{v}}. \label{reps}
\end{equation}
We expand the right-hand side of (\ref{reps}), cancel terms, and rearrange to write
\begin{equation}
\sum_{k=2}^{\infty}\frac{\mathcal{P}_{\chi_{0}}(2k\sigma)}{k^{v}}=-\sum_{n=2}^{\infty}\frac{\mu(n)}{n^{v}}\sum_{k=1}^{\infty}\frac{\mathcal{P}_{\chi_{0}}(2kn\sigma)}{k^{v}}. \label{useful}
\end{equation}
Dividing (\ref{useful}) by 2, setting $v=2$, comparing to (\ref{simp step 4}), and using (\ref{diag covar}) then gives (\ref{simp mse}).

We next consider (\ref{error 2}) and take the square, giving the series expression 
\begin{align}
\varepsilon^{2}_{\chi_{0}}\left(\sigma,t\right)=\sum_{n=2}^{\infty}\frac{\mu^{2}(n)}{n^{2}}\left(\log\left|L\left(n(\sigma+it),\chi_{0}\right)\right|\right)^{2} \nonumber \\ +2\sum_{n=2}^{\infty}\frac{\mu(n)}{n}\sum_{2\leq m<n}\frac{\mu(m)}{m}\left(\log\left|L\left(n(\sigma+it),\chi_{0}\right)\right|\right)\left(\log\left|L\left(m(\sigma+it),\chi_{0}\right)\right|\right). \label{comp mse step 1}
\end{align}
We note from the Euler product and Taylor expansion that, for any $n\sigma>1$, 
\begin{align}
\log\left|L\left(n\left(\sigma+it\right),\chi_{0}\right)\right|=\sum_{k=1}^{\infty}\frac{1}{k}\sum_{p \nmid M}\frac{\cos\left(nt\log p\right)}{p^{kn\sigma}} \nonumber \\ \leq \sum_{k=1}^{\infty}\frac{1}{k}\sum_{p \nmid M}\frac{1}{p^{kn\sigma}}=\log \left|L\left(n\sigma,\chi_{0}\right)\right| \label{log l bound}
\end{align}
and, as $n\rightarrow \infty$,
\begin{equation}
\log \left|L\left(n\sigma,\chi_{0} \right)\right| \sim \frac{1}{p^{n\sigma}}, \label{log l asy}
\end{equation}
where $p$ is the minimum prime such that $p \nmid M$. Hence the sum of the absolute values of the summands of (\ref{comp mse step 1})'s first series has the upper bound 
\begin{equation}
\sum_{n=2}^{\infty}\frac{1}{n^{2}}\left(\log\left|L\left(n\sigma,\chi_{0}\right)\right|\right)^{2}<\infty, \label{comp mse bound 1}
\end{equation}
while the sum of the absolute values of the summands of (\ref{comp mse step 1})'s second series has the upper bound
\begin{align}
2\sum_{n=2}^{\infty}\frac{1}{n}\sum_{2 \leq m <n}\frac{1}{m}\left(\log\left|L\left(n\sigma,\chi_{0}\right)\right|\right)\left(\log\left|L\left(m\sigma,\chi_{0}\right)\right|\right) \nonumber \\ <2\left(\log\left|L\left(\sigma,\chi_{0}\right)\right|\right)\sum_{n=2}^{\infty}\log\left|L\left(n\sigma,\chi_{0}\right)\right|<\infty. 
\label{comp mse bound 2}
\end{align}
The finite bounds (\ref{comp mse bound 1})-(\ref{comp mse bound 2}) show that (\ref{comp mse step 1}) satisfies the condition (\ref{ft cond}). We may therefore apply the expectation to (\ref{comp mse step 1}) term-by term. We do so and write 
\begin{align}
E\left\{\varepsilon^{2}_{\chi_{0}}\left(\sigma,t\right)\right\}=\sum_{n=2}^{\infty}\frac{\mu^{2}(n)}{n^{2}}R_{\chi_{0}}^{\textnormal{diag}}(n,n,\sigma)+2\sum_{n=2}^{\infty}\sum_{2\leq m<n}\frac{\mu(n)\mu(m)}{nm}R_{\chi_{0}}^{\textnormal{diag}}(n,m,\sigma). \label{mse raw}
\end{align}
We expand the second summation to give
\begin{align}
 E\left\{\varepsilon^{2}_{\chi_{0}}\left(\sigma,t\right)\right\}=\sum_{n=2}^{\infty}\frac{\mu^{2}(n)}{n^{2}}R_{\chi_{0}}^{\textnormal{diag}}(n,n,\sigma) \nonumber \\ +2\sum_{n=2}^{\infty}\frac{\mu(n)}{n}\left(\sum_{\substack{m|n \\2 \leq m<n}}\frac{\mu(m)}{m}R_{\chi_{0}}^{\textnormal{diag}}(n,m,\sigma)+\sum_{\substack{m \nmid n \\ m<n}}\frac{\mu(m)}{m}R_{\chi_{0}}^{\textnormal{diag}}(n,m,\sigma)\right). \label{mse raw expanded}
\end{align}
We then apply (\ref{logzeta cov alternate}) in (\ref{mse raw expanded}) and simplify to give the result
\begin{align}
E\left\{\varepsilon^{2}_{\chi_{0}}\left(\sigma,t\right)\right\}=\sum_{n=2}^{\infty}\frac{\mu^{2}(n)}{n^{2}}R_{\chi_{0}}^{\textnormal{diag}}(n,n,\sigma) \nonumber \\ +2\sum_{n=2}^{\infty}\frac{\mu(n)}{n^{2}}\left(\sum_{\substack{m|n \\ 2\leq m<n}}\mu(m)R_{\chi_{0}}^{\textnormal{diag}}(n,n,\sigma)+\sum_{\substack{m \nmid n \\ m<n}}\frac{\mu(m)}{m^{2}}R_{\chi_{0}}^{\textnormal{diag}}(nm,nm,\sigma)\right) \label{mse less raw}
\end{align}
We next gather the series in (\ref{mse less raw}) to write
\begin{align}
\begin{split}E\left\{\varepsilon^{2}_{\chi_{0}}\left(\sigma,t\right)\right\}=\sum_{n=2}^{\infty}\frac{\mu(n)}{n^{2}}\left[\left(\mu(n)+2\sum_{\substack{m|n\\ 2 \leq m<n}}\mu(m)\right)R_{\chi_{0}}^{\textnormal{diag}}(n,n,\sigma) \right. \\ \left.+2\sum_{\substack{m \nmid n\\ m<n}}\frac{\mu(m)}{m^{2}}R_{\chi_{0}}^{\textnormal{diag}}(nm,nm,\sigma)\right]\end{split} \label{mse still less raw}
\end{align}
We then note that 
\begin{equation}
\mu(n)+2\sum_{\substack{m|n \\ 2\leq m<n}}\mu(m)=\mu(n)+2\left(0-\mu(n)-1\right)=-\mu(n)-2 \label{basic mobius}
\end{equation}
by (\ref{mobius vanishing}), which enables us to simplify (\ref{mse still less raw}) to give (\ref{mse still still less raw}). We lastly make note of the constraints on $\sigma$ for the convergence of (\ref{diag covar})'s $R_{\chi_{0}}^{\textnormal{diag}}(n,n,\sigma)$. Since all of the summands in (\ref{simp mse})-(\ref{mse still still less raw}) involve $R_{\chi_{0}}^{\textnormal{diag}}(n,n,\sigma)$ with $n \geq 2$, we conclude from (\ref{diag covar})'s third line that (\ref{simp mse}) and (\ref{mse still still less raw}) are convergent for all $\sigma>1/4$. This completes the proof. 
\end{proof}
It is clear from the decay rates implied by (\ref{primedirseries}), (\ref{polylog}), and (\ref{diag covar}) that the summations in both (\ref{simp mse}) and (\ref{mse still still less raw}) rapidly converge for $\sigma>1/4$ and hence their results may be numerically approximated. Over the entire investigated range of $\sigma$, approximations using both (\ref{simp mse}) and (\ref{mse still still less raw}) give nearly equivalent numerical results. This is clear from Figure \ref{pzmse fig}, where we depict estimates from (\ref{simp mse}) and (\ref{mse still still less raw}) of (\ref{error})'s root mean square for $\chi=M=1$ and values of $1/4<\sigma\leq 3$. These estimates show that (\ref{error})'s root mean square with $\chi=M=1$ increases with decreasing $\sigma$, demonstrating the greater impact of nonlocal correlations on $\log\left|\zeta\left(\sigma+it\right)\right|$ with smaller $\sigma$. Specific estimates for this root mean square include $\approx .006$ near $\sigma=3$, $\approx .023$ near $\sigma=2$, $\approx.104$ near $\sigma=1$, and $\approx .264$ near $\sigma=1/2$.

\begin{figure}
\centering
\includegraphics[width=\linewidth]{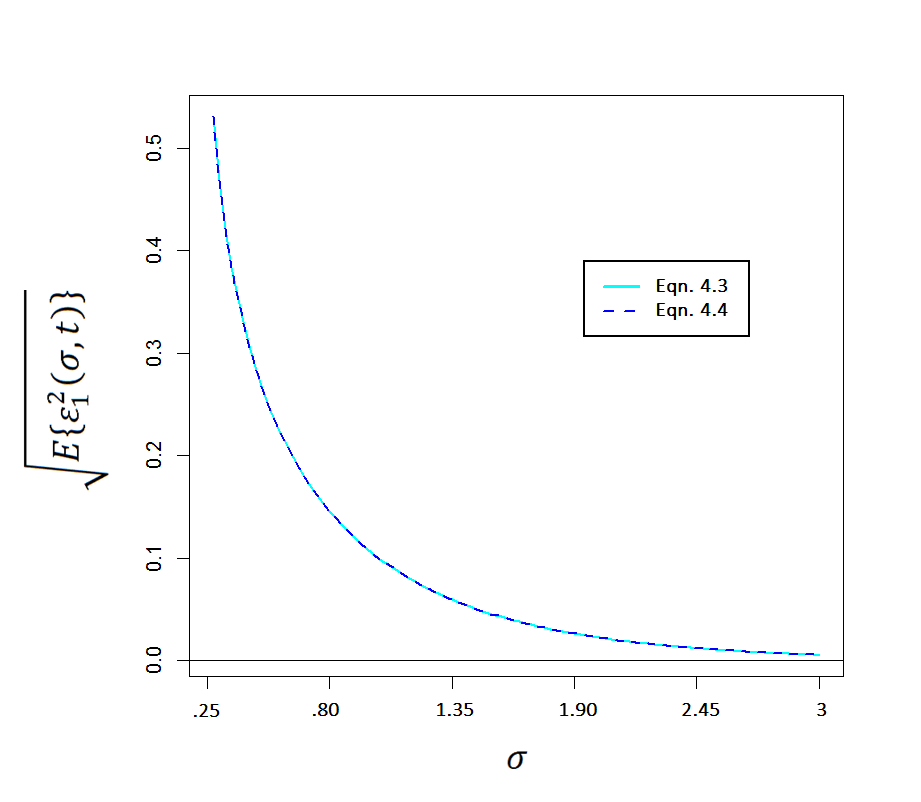}
\caption{A plot depicting the root mean square of (\ref{error}) with $\chi_{0}=M=1$, approximated using (\ref{simp mse}) and (\ref{mse still still less raw}) with (\ref{logzeta cov 2}) for $1/4< \sigma \leq 3$.}
\label{pzmse fig}
\end{figure}

\subsection{Proof of Theorem \ref{bigthm1}}
With Lemma \ref{mse thm} we are now ready to prove Theorem \ref{bigthm1}

\begin{proof}
We apply Lemma \ref{mse thm} and set (\ref{simp mse}) equal to (\ref{mse still still less raw}). Then gathering summations gives
\begin{equation}
\sum_{n=2}^{\infty}\frac{\mu(n)}{n^{2}}\left(1+\mu(n)\right)R_{\chi_{0}}^{\textnormal{diag}}(n,n,\sigma)=2\sum_{n=2}^{\infty}\frac{\mu(n)}{n^{2}}\sum_{\substack{m \nmid n\\ m<n}}\frac{\mu(m)}{m^{2}}R_{\chi_{0}}^{\textnormal{diag}}(nm,nm,\sigma) \label{oooo}
\end{equation}
We then note that $\mu(n)(1+\mu(n))$ is only nonzero when $\mu(n)=1$, in which case $\mu(n)(1+\mu(n))=2$. We additionally note that no $n \leq 2$ has any non-divisors $m<n$. We apply these facts to (\ref{oooo}) to give the identity 
\begin{equation}
\sum_{\substack{n>1 \\ \mu(n)=1}}\frac{1}{n^{2}}R_{\chi_{0}}^{\textnormal{diag}}(n,n,\sigma)=\sum_{n}\frac{\mu(n)}{n^{2}}\sum_{\substack{m \nmid n\\ m<n}}\frac{\mu(m)}{m^{2}}R_{\chi_{0}}^{\textnormal{diag}}(nm,nm,\sigma). \label{oooo 2}
\end{equation}

We next consider two principal Dirichlet characters $\chi_{0}$ and $\chi_{0}'$ with respective moduli $M$ and $M'$ such that, for some prime $p \nmid M$, $M'=pM$. Then by (\ref{diag covar})'s third line,
\begin{equation}
R_{\chi_{0}}^{\textnormal{diag}}(n,n,\sigma)-R_{\chi_{0}'}^{\textnormal{diag}}(n,n,\sigma)=\frac{1}{2}\textnormal{Li}_{2}\left(\frac{1}{p^{2n\sigma}}\right) \label{hellyeah}
\end{equation}
for all $\sigma>\frac{1}{2n}$ . We then apply (\ref{oooo 2}) for both $\chi_{0}$ and $\chi_{0}'$ and take the difference of the two results to give
\begin{align}
\sum_{\substack{n>1 \\ \mu(n)=1}}\frac{1}{n^{2}}\left(R_{\chi_{0}}^{\textnormal{diag}}(n,n,\sigma)-R_{\chi_{0}'}^{\textnormal{diag}}(n,n,\sigma)\right)\nonumber \\ =\sum_{n}\frac{\mu(n)}{n^{2}}\sum_{\substack{m \nmid n\\ m<n}}\frac{\mu(m)}{m^{2}}\left(R_{\chi_{0}}^{\textnormal{diag}}(nm,nm,\sigma)-R_{\chi_{0}'}^{\textnormal{diag}}(nm,nm,\sigma)\right). \label{oooo diff}
\end{align}
Then, by (\ref{hellyeah}), (\ref{oooo diff}) becomes
\begin{equation}
\sum_{\substack{n>1 \\ \mu(n)=1}}\frac{1}{n^{2}}\textnormal{Li}_{2}\left(\frac{1}{p^{2n\sigma}}\right)=\sum_{n}\frac{\mu(n)}{n^{2}}\sum_{\substack{m \nmid n\\ m<n}}\frac{\mu(m)}{m^{2}}\textnormal{Li}_{2}\left(\frac{1}{p^{2nm\sigma}}\right).\label{oooo final}
\end{equation}
We note that the the first nonzero terms in (\ref{oooo diff})-(\ref{oooo final})'s series are those corresponding to $n=6$ on the left-hand side and $n=3$ on the right-hand side. We also note that $p$ may be any prime, and, from (\ref{hellyeah})'s constraints on $\sigma$, $\sigma>1/12$. Applying these facts in (\ref{oooo final}) implies 
\begin{equation}
\sum_{\substack{n>1\\ \mu(n)=1}}\frac{1}{n^{2}}\textnormal{Li}_{2}\left(\frac{1}{x^{n}}\right)=\sum_{n}\frac{\mu(n)}{n^{2}}\sum_{\substack{m \nmid n\\ m<n}}\frac{\mu(m)}{m^{2}}\textnormal{Li}_{2}\left(\frac{1}{x^{nm}}\right)\label{oooo final 2}
\end{equation}
for all $x>2^{1/6}$.

We next show that both sides of (\ref{dabigone}) and hence both sides of (\ref{oooo final 2}) are uniformly continuous on $x\geq x_{0}$, where $x_{0}>1$ is arbitrary. We begin with the left-hand side of (\ref{dabigone}) and show that it is uniformly convergent on $x \geq x_{0}$ for any $v \in \mathbb{R}$ by noting that
\begin{equation}
\left|\sum_{\substack{n>N \\ \mu(n)=1}}\frac{1}{n^{v}}\textnormal{Li}_{v}\left(\frac{1}{x^{n}}\right)\right| < \sum_{n>N}\frac{1}{n^{v}}\textnormal{Li}_{v}\left(\frac{1}{x^{n}}\right) \leq  \sum_{n>N}\frac{1}{n^{v}}\textnormal{Li}_{v}\left(\frac{1}{x_{0}^{n}}\right). \label{easy uniform}
\end{equation}
Since, by (\ref{polylog}), the series on the far right-hand side is convergent for any finite $v \in \mathbb{R}$, (\ref{easy uniform}) shows uniform convergence of (\ref{dabigone})'s left-hand side. We then proceed to showing that the right-hand side of (\ref{dabigone}) is uniformly convergent on $x \geq x_{0}$ and $v\geq 1$ by noting that 
\begin{align}
\left|\sum_{n> N}\frac{\mu(n)}{n^{v}}\sum_{\substack{m \nmid n \\ m<n}}\frac{\mu(m)}{m^{v}}\textnormal{Li}_{v}\left(\frac{1}{x^{nm}}\right)\right| \leq \sum_{n> N}\frac{1}{n^{v}}\sum_{\substack{m \nmid n \\ m<n}}\frac{1}{m^{v}}\textnormal{Li}_{v}\left(\frac{1}{x^{nm}}\right) \nonumber \\
< \sum_{n> N}\frac{1}{n^{v}}\sum_{m\leq n}\textnormal{Li}_{v}\left(\frac{1}{x^{nm}}\right) < \sum_{n> N}\frac{1}{n^{v-1}}\textnormal{Li}_{v}\left(\frac{1}{x^{n}}\right) \leq \sum_{n> N}\frac{1}{n^{v-1}}\textnormal{Li}_{v}\left(\frac{1}{x_{0}^{n}}\right). \label{a little harder uniform}
\end{align}
Similarly to above, the convergence of the series on (\ref{a little harder uniform})'s lower-right proves uniform convergence of (\ref{dabigone})'s right-hand side for $v\geq 1$. We lastly show the uniform convergence of (\ref{dabigone})'s right-hand side on $x \geq x_{0}$ and $v\leq 0$ by noting
\begin{align}
\left|\sum_{n> N}\mu(n)n^{|v|}\sum_{\substack{m \nmid n \\ m<n}}\mu(m)m^{|v|}\textnormal{Li}_{v}\left(\frac{1}{x^{nm}}\right)\right| \leq \sum_{n>N}n^{|v|}\sum_{\substack{m \nmid n \\ m<n}}m^{|v|}\textnormal{Li}_{v}\left(\frac{1}{x^{nm}}\right) \label{hardest uniform} \\ \nonumber 
<\sum_{n> N}n^{2|v|}\sum_{m\leq n}\textnormal{Li}_{v}\left(\frac{1}{x^{nm}}\right)<\sum_{n>N}n^{2|v|+1}\textnormal{Li}_{v}\left(\frac{1}{x^{n}}\right)\leq \sum_{n> N}n^{2|v|+1}\textnormal{Li}_{v}\left(\frac{1}{x_{0}^{n}}\right)
\end{align}
Again, the convergence of the series on (\ref{hardest uniform})'s lower-right proves uniform convergence of (\ref{dabigone})'s right-hand side for $v\leq 0$. We have thus shown that both sides of (\ref{dabigone}) for all $v \in \mathbb{R}$ and hence both sides of (\ref{oooo final 2}) are uniformly convergent on $x \geq x_{0}$. Since $x_{0}>1$ is arbitrary, this implies that both sides of (\ref{dabigone}) for all $v \in \mathbb{R}$, and hence both sides of (\ref{oooo final 2}), are uniformly convergent on $x>2^{1/6}$.

We now suppose that (\ref{dabigone}) is true for some $v \in \mathbb{R}$. We then consider the integral 
\begin{align}
\int_{\log x}^{\infty}\sum_{\substack{n>1 \\ \mu(n)=1}}\frac{1}{n^{v}}\textnormal{Li}_{v}\left(\frac{1}{u^{n}}\right)du =\int_{\log x}^{\infty}\sum_{n}\frac{\mu(n)}{n^{v}}\sum_{\substack{m \nmid n\\ m<n}}\frac{\mu(m)}{m^{v}}\textnormal{Li}_{v}\left(\frac{1}{u^{nm}}\right)du \label{oooo int step 1}
\end{align}
with $x>2^{1/6}$. Since both sides of (\ref{dabigone}) are uniformly convergent on $x>2^{1/6}$, we may integrate both sides of (\ref{oooo int step 1}) term-by-term. We then note the following integral of the polylogarithm:
\begin{equation}
\int_{\log x}^{\infty}\textnormal{Li}_{v}\left(\frac{1}{u^{\alpha}}\right)du=\frac{1}{\alpha}\textnormal{Li}_{v+1}\left(\frac{1}{x^{\alpha}}\right) \label{li int}
\end{equation}
Applying (\ref{li int}) with $\alpha=n$ on (\ref{oooo int step 1})'s left-hand side and applying (\ref{li int}) with $\alpha=nm$ on (\ref{oooo int step 1})'s right-hand side gives
\begin{equation}
\sum_{\substack{n >1 \\ \mu(n)=1}}\frac{1}{n^{v+1}}\textnormal{Li}_{v+1}\left(\frac{1}{x^{n}}\right)=\sum_{n}\frac{\mu(n)}{n^{v+1}}\sum_{\substack{m \nmid n\\ m<n}}\frac{\mu(m)}{m^{v+1}}\textnormal{Li}_{v+1}\left(\frac{1}{x^{nm}}\right). \label{oooo int step 2}
\end{equation}
Therefore (\ref{dabigone}) implies (\ref{oooo int step 2}). This fact combined with (\ref{oooo final 2}) proves (\ref{dabigone}) for all $v \in \mathbb{Z}$ such that $v\geq 2$ on $x>2^{1/6}$.

We again suppose (\ref{dabigone}) is true for some $v \in \mathbb{R}$, set $u=\log x$ with $x>2^{1/6}$ and consider the derivative
\begin{equation}
\frac{\partial}{\partial u}\sum_{\substack{n >1 \\ \mu(n)=1}}\frac{1}{n^{v}}\textnormal{Li}_{v}\left(\frac{1}{e^{nu}}\right)=\frac{\partial}{\partial x}\sum_{n}\frac{\mu(n)}{n^{v}}\sum_{\substack{m \nmid n\\ m<n}}\frac{\mu(m)}{m^{v}}\textnormal{Li}_{v}\left(\frac{1}{e^{nmu}}\right).
\label{oooo deriv step 1}
\end{equation}
We differentiate both sides of (\ref{oooo deriv step 1}) term-by-term and note the following derivative of the polylogarithm:
\begin{equation}
\frac{\partial}{\partial u}\textnormal{Li}_{v}\left(\frac{1}{e^{\alpha u}}\right)=-\alpha \textnormal{Li}_{v-1}\left(\frac{1}{e^{\alpha u}}\right) \label{li deriv}
\end{equation}
Applying (\ref{li deriv}) with $\alpha=n$ on (\ref{oooo deriv step 1})'s left-hand side, applying (\ref{li deriv}) with $\alpha=nm$ on (\ref{oooo deriv step 1})'s right-hand side, and resubstituting $\log x$ for $u$ gives
\begin{equation}
\sum_{\substack{n >1 \\ \mu(n)=1}}\frac{1}{n^{v-1}}\textnormal{Li}_{v-1}\left(\frac{1}{x^{n}}\right)=\sum_{n}\frac{\mu(n)}{n^{v-1}}\sum_{\substack{m \nmid n\\ m<n}}\frac{\mu(m)}{m^{v-1}}\textnormal{Li}_{v-1}\left(\frac{1}{x^{nm\sigma}}\right). \label{oooo deriv step 2}
\end{equation}
Since both sides of (\ref{oooo deriv step 2}) are uniformly convergent on $x>2^{1/6}$, term-by-term differentiation to evaluate (\ref{oooo deriv step 1}) was valid. Therefore (\ref{dabigone}) implies (\ref{oooo deriv step 2}). This fact combined with (\ref{oooo final 2}) proves (\ref{dabigone}) for all $v \in \mathbb{Z}$ such that $v \leq 2$ on $x>2^{1/6}$. This proves (\ref{dabigone}) for all $v\in \mathbb{Z}$ with $x>2^{1/6}$.

We next note that $\textnormal{Li}_{v}(z)$ has the property
\begin{equation}
\frac{\partial^{l}}{\partial v^{l}}\textnormal{Li}_{v}\left(z\right)=\left(-1\right)^{l}\sum_{k=1}^{\infty}\frac{\log^{l}(k)}{k^{v}}z^{k}.
\end{equation}
We then show that 
\begin{equation}
\sum_{\substack{n>1 \\ \mu(n)=1}}\frac{1}{n^{v}}\sum_{k=1}^{\infty}\frac{\log^{l}(k)}{k^{v}x^{kn}} \label{new step lhs}
\end{equation}
is uniformly convergent on $v\geq 0$ for any  $l \in \mathbb{R}$ and $x>2^{1/6}$ by noting that
\begin{equation}
\sum_{\substack{n>N \\ \mu(n)=1}}\frac{1}{n^{v}}\sum_{k=1}^{\infty}\frac{\log^{l}(k)}{k^{v}x^{kn}}<\sum_{n>N}\frac{1}{n^{v}}\sum_{k=1}^{\infty}\frac{\log^{l}(k)}{k^{v}x^{kn}}\leq \sum_{n>N}\sum_{k=1}^{\infty}\frac{\log^{l}(k)}{x^{kn}}. \label{new step lhs uc}
\end{equation}
We similarly show that 
\begin{equation}
\sum_{n}\frac{\mu(n)}{n^{v}}\sum_{\substack{m \nmid n \\ m<n}}\frac{\mu(m)}{m^{v}}\sum_{k=1}^{\infty}\frac{\log^{l}(k)}{k^{v}x^{knm}}
\label{new step rhs}
\end{equation}
is uniformly convergent on $v\geq 0$ for any $l \in \mathbb{R}$ and $x>2^{1/6}$ by noting
\begin{align}
\sum_{n>N}\frac{1}{n^{v}}\sum_{\substack{m \nmid n \\ m<n}}\frac{1}{m^{v}}\sum_{k=1}^{\infty}\frac{\log^{l}(k)}{k^{v}x^{knm}}\leq \sum_{n>N}\sum_{m<n}\sum_{k=1}^{\infty}\frac{\log^{l}(k)}{x^{knm}}<\sum_{n>N}n\sum_{k=1}^{\infty}\frac{\log^{l}(k)}{x^{kn}}.
\label{new step rhs uc}
\end{align}

By the uniform convergence of (\ref{new step lhs}) and (\ref{new step rhs}) we may apply the Taylor series expansion to write
\begin{equation}
\sum_{\substack{n>1 \\ \mu(n)=1}}\frac{1}{n^{v_{b}}}\textnormal{Li}_{v_{b}}\left(\frac{1}{x^{n}}\right)=\sum_{l=0}^{\infty}\frac{(-1)^{l}}{l!}\left(v_{b}-v_{a}\right)^{l}\sum_{\substack{n>1 \\ \mu(n)=1}}\frac{1}{n^{v_{a}}}\sum_{k=1}^{\infty}\frac{\log^{l}(k)}{k^{v_{a}}x^{kn}} \label{pwr series lhs}
\end{equation}
and
\begin{align}
\sum_{n}\frac{\mu(n)}{n^{v_{b}}}\sum_{\substack{m \nmid n \\ m<n}}\frac{\mu(m)}{m^{v_{b}}}\textnormal{Li}_{v_{b}}\left(\frac{1}{x^{nm}}\right) \nonumber \\ =\sum_{l=0}^{\infty}\frac{(-1)^{l}}{l!}\left(v_{b}-v_{a}\right)^{l}\sum_{n}\frac{\mu(n)}{n^{v_{a}}}\sum_{\substack{m\nmid n \\m<n}}\frac{\mu(m)}{m^{v_{a}}}\sum_{k=1}^{\infty}\frac{\log^{l}(k)}{k^{v_{a}}x^{knm}}\label{pwr series rhs}
\end{align}
for some $v_{a},v_{b}\in \mathbb{Z}$. Since (\ref{dabigone}) with $v\in \mathbb{Z}$ is true, (\ref{pwr series lhs}) and (\ref{pwr series rhs}) are equal, which implies the equality of the $l$th summands in (\ref{pwr series lhs}) and (\ref{pwr series rhs}), proving that 
\begin{equation}
\sum_{\substack{n>1 \\ \mu(n)=1}}\frac{1}{n^{v}}\sum_{k=1}^{\infty}\frac{\log^{l}(k)}{k^{v}x^{kn}}=\sum_{n}\frac{\mu(n)}{n^{v}}\sum_{\substack{m \nmid n \\ m<n}}\frac{\mu(m)}{m^{v}}\sum_{k=1}^{\infty}\frac{\log^{l}(k)}{k^{v}x^{knm}}
\label{new step done}
\end{equation}
for all $v\in \mathbb{Z}$, $l \in \mathbb{N}$, and $x>2^{1/6}$. We then note that one may apply (\ref{pwr series lhs})-(\ref{pwr series rhs}) for any $v_{a} \in \mathbb{Z}$ and any $v_{b} \in \mathbb{R}$, which, with (\ref{new step done}), shows that (\ref{dabigone}) holds for all $v \in \mathbb{R}$ and $x>2^{1/6}$. We lastly note that, by (\ref{gen x}), (\ref{dabigone})'s left-hand side satisfies
\begin{align}
\sum_{\substack{n >1 \\ \mu(n)=1}}\frac{1}{n^{v}}\textnormal{Li}_{v}\left(\frac{1}{x^{n}}\right)=\frac{1}{x}-\sum_{\substack{n>1 \\ \mu(n) \neq 1}}\frac{\mu(n)}{n^{v}}\textnormal{Li}_{v}\left(\frac{1}{x^{n}}\right)-\textnormal{Li}_{v}\left(\frac{1}{x}\right) \nonumber \\
=\frac{1}{x}-\textnormal{Li}_{v}\left(\frac{1}{x}\right)+\sum_{\substack{n>1 \\ \mu(n)=-1}}\frac{1}{n^{v}}\textnormal{Li}_{v}\left(\frac{1}{x^{n}}\right). \label{gen y}
\end{align}
Setting the last line of (\ref{gen y}) equal to (\ref{dabigone})'s right-hand side and rearranging proves (\ref{dabigone2}).
\end{proof}

We close by noting that the major insight enabling Theorem \ref{bigthm1}'s proof is the result (\ref{oooo 2}), which describes an important symmetry between resonant and dissonant nonlocal correlations in $\log\left|L\left(s,\chi_{0}\right)\right|$. This symmetry involves sums over integers with an even number of distinct prime factors and sums over the non-divisors of the integers, leading to the general result (\ref{dabigone}). The result (\ref{dabigone2}) is then an immediate corollary of (\ref{dabigone}). Potential directions for future research include the study of correlation structures in non-principal L-functions as well as the study of higher order correlation structures, e.g., three or four-point correlations, in $\textnormal{Re}\mathcal{P}_{\chi}(s)$ and $\log\left|L\left(s,\chi\right)\right|$.

\appendix

\section{On statistical independence}
\label{app independence}
The characteristic function of a random variable $X$ is defined
\begin{equation}
\varphi_{X}(\lambda)=E\left\{e^{i\lambda X}\right\}. \label{chf}
\end{equation} 
We suppose $t$ is uniformly distributed in $[a,b]$ with $b-a \rightarrow \infty$ and consider the following sum over some set of primes $p$: 
\begin{equation}
X(t)=\sum_{p}a_{p}e^{-i\left(\alpha_{p}t\log p+\theta_{p}\right)}, \label{primedirseries m}
\end{equation}
where $a_{p}, \theta_{p} \in \mathbb{R}$ and $\alpha_{p} \in \mathbb{Q}$. We substitute (\ref{primedirseries m})'s real part into (\ref{chf}) to write
\begin{equation}
\varphi_{X(t)}(\lambda)=E\left\{ \prod_{p}\exp\left( i \lambda a_{p}\cos\left(\alpha_{p}t \log p+\theta_{p}\right)\right) \right\}. \label{app1step00}
\end{equation}
We then expand (\ref{app1step00}) using the Bessel function identity 
\begin{equation}
e^{ix \cos y}=\sum_{n=-\infty}^{\infty} i^{n}J_{n}\left(x\right)e^{i n y}
\label{besselidcos}
\end{equation}
where $J_{n}(.)$ is the $n$th-order Bessel function of the first kind. This gives
\begin{align}
\begin{split}
\varphi_{X(t)}\left(\lambda \right)=E\left\{\sum_{n_{1},n_{2},...}\left(i^{n_{1}+n_{2}+...}J_{n_{1}}\left(\lambda a_{p_{1}}\right)J_{n_{2}}\left(\lambda a_{p_{2}}\right)... \right. \right. \\ \left. \left. \times e^{i\left(n_{1}\theta_{p_{1}}+n_{2}\theta_{p_{2}}+...\right)}e^{i t\left(n_{1}\alpha_{p_{1}}\log p_{1}+n_{2}\alpha_{p_{2}}\log p_{2}+...\right)}\right)\right\} \end{split}
\label{big}
\end{align}
The exponential terms on (\ref{big})'s far right-hand side are unit circle rotations with $t$. Therefore taking the expected value will cause all terms to vanish except those for which 
\begin{equation}
n_{1}\alpha_{p_{1}}\log p_{1}+n_{2}\alpha_{p_{2}}\log p_{2}+n_{3}\alpha_{p_{3}}\log p_{3}+...=0. \label{zero}
\end{equation}
However, by unique-prime-factorization, the $\log p$'s are linearly independent over the rational numbers. Therefore the only solution to (\ref{zero}) is given by 
\begin{equation}
n_{1}=n_{2}=n_{3}=...=0. \label{soln}
\end{equation}
This simplifies (\ref{big}) to give 
\begin{equation}
\varphi_{X(t)}\left(\lambda \right)=\prod_{p} J_{0}\left(\lambda a_{p}\right). \label{result}
\end{equation}
We next note from (\ref{chf}) and (\ref{besselidsin}) that the characteristic function for a single summand in (\ref{primedirseries m})'s real part, $a_{p}\cos(\alpha_{p}t \log p+\theta_{p})$, is given by 
\begin{equation}
\varphi_{p}\left(\lambda\right)=J_{0}\left(\lambda a_{p}\right). \label{littleresult}
\end{equation}
Therefore, by (\ref{result}) and (\ref{littleresult}), 
\begin{equation}
\varphi_{X(t)}\left(\lambda \right)=\prod_{p} \varphi_{p}\left(\lambda\right). \label{independence}
\end{equation}
This shows that (\ref{primedirseries m})'s summands are independent. Essentially equivalent reasoning using the identity
\begin{equation}
e^{ix \sin \phi}=\sum_{n=-\infty}^{\infty}J_{n}\left(x\right)e^{i n\phi} \label{besselidsin}
\end{equation}
gives equivalent results for (\ref{primedirseries m})'s imaginary part. This completes the proof sketch. For further details and complete proof, apply results in \cite{laurincikas} (pg. 36, Cor. 5.2 pg. 38, Eqn. 6.4 pg. 41, Cor. 6.7 pg. 43, Cor. 6.8 pg. 44).

\end{document}